\documentclass[11pt]{article}
\usepackage{a4}
\usepackage{amssymb}
\usepackage{amsmath}
\usepackage{amsthm}
\usepackage{enumerate}
\usepackage[pagebackref,colorlinks,citecolor=blue,linkcolor=blue]{hyperref}

\usepackage{geometry}
\numberwithin{equation}{section}
\theoremstyle{plain}
\newtheorem{theorem}[equation]{Theorem}
\newtheorem{corollary}[equation]{Corollary}
\newtheorem{lemma}[equation]{Lemma}
\newtheorem{proposition}[equation]{Proposition}

\theoremstyle{definition}
\newtheorem{definition}[equation]{Definition}

\newtheorem{example}[equation]{Example}

\numberwithin{equation}{section}

\newcommand{\R}{{\mathbb R}}
\newcommand{\N}{{\mathbb N}}

\newcommand{\Om}{\Omega}

\providecommand{\vint}[1]{\mathchoice
          {\mathop{\vrule width 5pt height 3 pt depth -2.5pt
                  \kern -9pt \kern 1pt\intop}\nolimits_{\kern -5pt{#1}}}
          {\mathop{\vrule width 5pt height 3 pt depth -2.6pt
                  \kern -6pt \intop}\nolimits_{\kern -3pt{#1}}}
          {\mathop{\vrule width 5pt height 3 pt depth -2.6pt
                  \kern -6pt \intop}\nolimits_{\kern -3pt{#1}}}
          {\mathop{\vrule width 5pt height 3 pt depth -2.6pt
                  \kern -6pt \intop}\nolimits_{\kern -3pt{#1}}}}

\newcommand{\eps}{\varepsilon}
\newcommand{\loc}{\mathrm{loc}}

\newcommand{\BV}{\mathrm{BV}}
\newcommand{\SBV}{\mathrm{SBV}}
\newcommand{\DBV}{\mathrm{DBV}}
\newcommand{\liploc}{\mathrm{Lip}_{\mathrm{loc}}}

\newcommand{\ch}{\text{\raise 1.3pt \hbox{$\chi$}\kern-0.2pt}}

\DeclareMathOperator{\capa}{Cap}
\DeclareMathOperator{\rcapa}{cap}

\DeclareMathOperator{\Lip}{Lip}

\DeclareMathOperator{\lip}{lip}
\DeclareMathOperator{\supp}{spt}

\DeclareMathOperator{\fint}{fine-int}

\begin{document}
\title{Approximation of $\BV$ by $\SBV$ functions\\
in metric spaces
\footnote{{\bf 2010 Mathematics Subject Classification}: 30L99, 31E05, 26B30.
\hfill \break {\it Keywords\,}: metric measure space, special function of bounded variation, strict convergence, uniform approximation, jump set, variational capacity
}}
\author{Panu Lahti}
\maketitle

\begin{abstract}
In a complete metric space that is equipped with a doubling measure and
supports a Poincar\'e
inequality, we show that functions of bounded variation ($\BV$ functions) can be
approximated in the strict sense
and pointwise uniformly by special functions of bounded variation, without adding
significant jumps.
As a main tool, we study the variational $1$-capacity and its $\BV$ analog.
\end{abstract}

\section{Introduction}

In the theory of functions of bounded variation,
one is often interested in approximating a BV function
by more regular functions, see e.g. \cite{ADC,BCP,CT,KR}.
Already the definition of the total variation
in metric spaces is based on such approximations.
The variation measure of a BV function can be decomposed into three parts:
the absolutely continuous part, the Cantor part, and the jump part.
Of these, the absolutely continuous part is of the same dimension as the space,
and the jump part is of dimension one less than the space.
The Cantor part can be of any dimension between these, making it often
more difficult to analyze than the other two parts.
A function is said to be in the SBV class (special functions
of bounded variation), first introduced in \cite{ADG},
if its variation measure has no Cantor part.
The recent paper \cite{dPFP} (as well as the earlier papers
\cite{BCP,CT}) studied how SBV function in Euclidean spaces
can be approximated in the BV norm by piecewise smooth functions.
This is based on the fact that outside its jump set,
an SBV function is essentially a Sobolev function, and then it is possible
to construct convolution approximations that
are close to the original function in the Sobolev/BV norm.

Due to the lack of structure of the Cantor part, it is in some way a rather
more subtle problem to approximate a general BV function by
SBV functions,
and little seems to be known in this direction.
It is impossible to find such approximations in the BV norm
(see Example \ref{ex:sharpness of BV norm closeness})
but we show in this paper that such approximations can be obtained in the following sense;
this will be given (with some more details) in Corollary \ref{cor:approximation result}.

\begin{theorem}\label{thm:main}
Let $\Om$ be an open set and let $u\in\BV(\Om)$. Then there exists a sequence $(u_i)\subset \SBV(\Om)$
such that
\begin{itemize}
\item $u_i\to u$ in $L^1(\Om)$ and $\Vert Du_i\Vert(\Om)\to \Vert Du\Vert(\Om)$,
\item $u_i\to u$ uniformly in $\Om$,
\item $\mathcal H(S_{u_i}\setminus S_u)=0$ for all $i\in\N$.
\end{itemize}
\end{theorem}

The first condition is often expressed by saying that ($u_i$) converges to
$u$ \emph{strictly} in $\BV(\Om)$. The last condition expresses the fact that
the approximation procedure does not add any significant \emph{jump set};
see Section \ref{sec:preliminaries} for definitions.
It is also possible to ensure that $u_i\ge u$ and that the $u_i$'s have the same
``boundary values'' as $u$.
Thus our result shows that it is sufficient to infimize various functionals
defined for the BV class, possibly involving also boundary values and obstacles,
only over the SBV class.
We discuss some implications of the result at the end of the paper.

In order to prove the approximation result, we first study some properties
of a class of BV functions with zero boundary values, which was previously studied
in \cite{L-ZB}. This is done in Section \ref{sec:BV0 class}.
Then in Section \ref{sec:capacities} we establish
the key tool needed for the approximation result,
namely a result on capacities that should be also of
independent interest.
The \emph{variational $p$-capacity} $\rcapa_p$ is an essential concept in
nonlinear potential theory, see e.g. the monographs
\cite{BB,HKM,MZ}.
In the case $p=1$, it is natural to also consider the $\BV$ analog $\rcapa_{\BV}$
of the variational $1$-capacity, and such a notion has been studied in the metric setting
in \cite{HaSh,KKST-DG,L-Fed}.
In \cite{HaSh} the authors considered a slightly different definition of this capacity
compared to ours,
but nonetheless it follows from
\cite[Theorem 4.3, Corollary 4.7]{HaSh} that
\[
\rcapa_{\BV}(A,D)\simeq\rcapa_{\lip,1}(A,D)
\]
when $A$ is a compact subset of an open set $D$,
and $\rcapa_{\lip,1}$ is a Lipschitz version of the variational $1$-capacity.
By the sign ``$\simeq$'' we mean that the quantities are comparable, with constants
of comparison depending only on the space.
In \cite[Theorem 4.23]{L-ZB} it was then shown that in fact equality holds.
In particular, this implies that
\[
\rcapa_{\BV}(A,D)=\rcapa_{1}(A,D).
\]
In this paper we show that this equality holds much more generally,
namely whenever $A$ is a \emph{quasiclosed} set and $D$ is a \emph{quasiopen} set.
This is given in Theorem \ref{thm:1capacity and BVcapacity are equal}.

Recently, there has been much interest in studying $\BV$ functions
and other topics of analysis in the abstract setting
of metric measure spaces, see e.g. \cite{A1,AMP,M}.
The standard assumptions in this setting
are that $(X,d,\mu)$ is a complete metric space equipped with
\emph{doubling} Radon measure $\mu$, and that the space supports a
\emph{Poincar\'e inequality}. While our results seem to be mostly new even in Euclidean spaces,
in this paper we also work in such a metric space setting.

\section{Preliminaries}\label{sec:preliminaries}

In this section we introduce the definitions,
assumptions, and some standard background results used in the paper.

Throughout this paper, $(X,d,\mu)$ is a complete metric space that is equip\-ped
with a metric $d$ and a Borel regular outer measure $\mu$ satisfying
a doubling property, meaning that
there exists a constant $C_d\ge 1$ such that
\[
0<\mu(B(x,2r))\le C_d\mu(B(x,r))<\infty
\]
for every ball $B(x,r):=\{y\in X:\,d(y,x)<r\}$.
Given a ball $B=B(x,r)$ and $\beta>0$, we sometimes abbreviate $\beta B:=B(x,\beta r)$.
When we want to state that a constant $C$
depends on the parameters $a,b, \ldots$, we write $C=C(a,b,\ldots)$.
When a property holds outside a set of $\mu$-measure zero, we say that it holds
almost everywhere, abbreviated a.e.

All functions defined on $X$ or its subsets will take values in $[-\infty,\infty]$.
A complete metric space equipped with a doubling measure is proper,
that is, closed and bounded sets are compact.
Given a $\mu$-measurable set $A\subset X$, we define $L^1_{\loc}(A)$ as the class
of functions $u$ on $A$
such that for every $x\in A$ there exists $r>0$ such that $u\in L^1(A\cap B(x,r))$.
Other local spaces of functions are defined similarly.
For an open set $\Omega\subset X$,
a function is in the class $L^1_{\loc}(\Omega)$ if and only if it is in $L^1(\Om')$ for
every open $\Omega'\Subset\Omega$.
Here $\Omega'\Subset\Omega$ means that $\overline{\Omega'}$ is a
compact subset of $\Omega$.

For any  $0<R<\infty$, the codimension one
Hausdorff content of a set $A\subset X$ is
\[
\mathcal{H}_{R}(A):=\inf\left\{ \sum_{i=1}^{\infty}
\frac{\mu(B(x_{i},r_{i}))}{r_{i}}:\,A\subset\bigcup_{i=1}^{\infty}B(x_{i},r_{i}),\,r_{i}\le R\right\}.
\]
The codimension one Hausdorff measure is then defined as
\[
\mathcal{H}(A):=\lim_{R\rightarrow 0}\mathcal{H}_{R}(A).
\]

By a curve we mean a nonconstant rectifiable continuous mapping from a compact interval of the real line into $X$.
The length of a curve $\gamma$
is denoted by $\ell_{\gamma}$. We will assume every curve to be parametrized
by arc-length, which can always be done (see e.g. \cite[Theorem~3.2]{Hj}).
A nonnegative Borel function $g$ on $X$ is an upper gradient 
of a function $u$
on $X$ if for all curves $\gamma$, we have
\begin{equation}\label{eq:definition of upper gradient}
|u(x)-u(y)|\le \int_0^{\ell_{\gamma}} g(\gamma(s))\,ds,
\end{equation}
where $x$ and $y$ are the end points of $\gamma$.
We interpret $|u(x)-u(y)|=\infty$ whenever  
at least one of $|u(x)|$, $|u(y)|$ is infinite.
Upper gradients were originally introduced in \cite{HK}.

We say that a family of curves $\Gamma$ is of zero $1$-modulus if there is a 
nonnegative Borel function $\rho\in L^1(X)$ such that 
for all curves $\gamma\in\Gamma$, the curve integral $\int_\gamma \rho\,ds$ is infinite.
A property is said to hold for $1$-almost every curve
if it fails only for a curve family with zero $1$-modulus. 
If $g$ is a nonnegative $\mu$-measurable function on $X$
and (\ref{eq:definition of upper gradient}) holds for $1$-almost every curve,
we say that $g$ is a $1$-weak upper gradient of $u$. 
By only considering curves $\gamma$ in $A\subset X$,
we can talk about a function $g$ being a ($1$-weak) upper gradient of $u$ in $A$.

Given a $\mu$-measurable set $H\subset X$, we let
\[
\Vert u\Vert_{N^{1,1}(H)}:=\Vert u\Vert_{L^1(H)}+\inf \Vert g\Vert_{L^1(H)},
\]
where the infimum is taken over all $1$-weak upper gradients $g$ of $u$ in $H$.
The substitute for the Sobolev space $W^{1,1}$ in the metric setting is the Newton-Sobolev space
\[
N^{1,1}(H):=\{u:\|u\|_{N^{1,1}(H)}<\infty\},
\]
which was first introduced in \cite{S}.
We understand a Newton-Sobolev function to be defined at every $x\in H$
(even though $\Vert \cdot\Vert_{N^{1,1}(H)}$ is then only a seminorm).
It is known that for any $u\in N_{\loc}^{1,1}(H)$ there exists a minimal $1$-weak
upper gradient of $u$ in $H$, always denoted by $g_{u}$, satisfying $g_{u}\le g$ 
a.e. in $H$ for any $1$-weak upper gradient $g\in L_{\loc}^{1}(H)$
of $u$ in $H$, see \cite[Theorem 2.25]{BB}.

We will assume throughout the paper that $X$ supports a $(1,1)$-Poincar\'e inequality,
meaning that there exist constants $C_P>0$ and $\lambda \ge 1$ such that for every
ball $B(x,r)$, every $u\in L^1_{\loc}(X)$,
and every upper gradient $g$ of $u$,
we have
\[
\vint{B(x,r)}|u-u_{B(x,r)}|\, d\mu 
\le C_P r\vint{B(x,\lambda r)}g\,d\mu,
\]
where 
\[
u_{B(x,r)}:=\vint{B(x,r)}u\,d\mu :=\frac 1{\mu(B(x,r))}\int_{B(x,r)}u\,d\mu.
\]

The $1$-capacity of a set $A\subset X$ is defined by
\[
\capa_1(A):=\inf \Vert u\Vert_{N^{1,1}(X)},
\]
where the infimum is taken over all functions $u\in N^{1,1}(X)$ such that $u\ge 1$ on $A$.
We know that $\capa_1$ is an outer capacity, meaning that
\[
\capa_1(A)=\inf\{\capa_1(W):\,W\supset A\textrm{ is open}\}
\]
for any $A\subset X$, see e.g. \cite[Theorem 5.31]{BB}.

If a property holds outside a set
$A\subset X$ with $\capa_1(A)=0$, we say that it holds $1$-quasieverywhere, or $1$-q.e.
If $H\subset X$ is $\mu$-measurable, then 
\begin{equation}\label{eq:quasieverywhere equivalence classes}
u=v\ \textrm{ 1-q.e. implies }\ \Vert u-v\Vert_{N^{1,1}(H)}=0,
\end{equation}
see \cite[Proposition 1.61]{BB}.

By \cite[Theorem 4.3, Theorem 5.1]{HaKi}, we know that for $A\subset X$,
\begin{equation}\label{eq:null sets of Hausdorff measure and capacity}
\capa_{1}(A)=0\quad
\textrm{if and only if}\quad\mathcal H(A)=0.
\end{equation}

Next we recall the definition and basic properties of functions
of bounded variation on metric spaces, following \cite{M}.
See also the monographs \cite{AFP, EvaG92, Fed, Giu84, Zie89} for the classical 
theory in the Euclidean setting.
We will always denote by $\Om\subset X$ an open set. Given a function $u\in L^1_{\loc}(\Om)$,
we define the total variation of $u$ in $\Om$ by
\[
\|Du\|(\Om):=\inf\left\{\liminf_{i\to\infty}\int_\Om g_{u_i}\,d\mu:\, u_i\in N^{1,1}_{\loc}(\Om),\, u_i\to u\textrm{ in } L^1_{\loc}(\Om)\right\},
\]
where each $g_{u_i}$ is the minimal $1$-weak upper gradient of $u_i$
in $\Om$.
(In \cite{M}, local Lipschitz constants were used in place of upper gradients, but the theory
can be developed similarly with either definition.)
We say that a function $u\in L^1(\Om)$ is of bounded variation, 
and denote $u\in\BV(\Om)$, if $\|Du\|(\Om)<\infty$.
For an arbitrary set $A\subset X$, we define
\[
\|Du\|(A):=\inf\{\|Du\|(W):\, A\subset W,\,W\subset X
\text{ is open}\}.
\]
In general, we understand the expression $\Vert Du\Vert(A)<\infty$ to mean that
there exists some open set $\Om\supset A$ such that $u$ is defined on $\Om$ with $u\in L^1_{\loc}(\Om)$
and $\Vert Du\Vert(\Om)<\infty$.
If $u\in L^1_{\loc}(\Om)$ and $\Vert Du\Vert(\Omega)<\infty$, $\|Du\|(\cdot)$ is
a Radon measure on $\Omega$ by \cite[Theorem 3.4]{M}.
A $\mu$-measurable set $E\subset X$ is said to be of finite perimeter if $\|D\ch_E\|(X)<\infty$, where $\ch_E$ is the characteristic function of $E$.
The perimeter of $E$ in $\Omega$ is also denoted by
\[
P(E,\Omega):=\|D\ch_E\|(\Omega).
\]
The $\BV$ norm is defined by
\[
\Vert u\Vert_{\BV(\Om)}:=\Vert u\Vert_{L^1(\Om)}+\Vert Du\Vert(\Om).
\]
The measure-theoretic interior of a set $E\subset X$ is defined by
\[
I_E:=
\left\{x\in X:\,\lim_{r\to 0}\frac{\mu(B(x,r)\setminus E)}{\mu(B(x,r))}=0\right\},
\]
and the measure-theoretic exterior by
\[
O_E:=
\left\{x\in X:\,\lim_{r\to 0}\frac{\mu(B(x,r)\cap E)}{\mu(B(x,r))}=0\right\}.
\]
The measure-theoretic boundary $\partial^{*}E$ is defined as the set of points
$x\in X$
at which both $E$ and its complement have strictly positive upper density, i.e.
\begin{equation}\label{eq:measure theoretic boundary}
\limsup_{r\to 0}\frac{\mu(B(x,r)\cap E)}{\mu(B(x,r))}>0\quad
\textrm{and}\quad\limsup_{r\to 0}\frac{\mu(B(x,r)\setminus E)}{\mu(B(x,r))}>0.
\end{equation}
Given an open set $\Omega\subset X$ and a $\mu$-measurable set $E\subset X$ with $P(E,\Omega)<\infty$, we know that for any Borel set $A\subset\Omega$,
\begin{equation}\label{eq:def of theta}
P(E,A)=\int_{\partial^{*}E\cap A}\theta_E\,d\mathcal H,
\end{equation}
where
$\theta_E\colon X\to [\alpha,C_d]$ with $\alpha=\alpha(C_d,C_P,\lambda)>0$, see \cite[Theorem 5.3, Theorem 5.4]{A1} 
and \cite[Theorem 4.6]{AMP}.
The following coarea formula is given in \cite[Proposition 4.2]{M}:
if $\Omega\subset X$ is an open set and $u\in \BV(\Omega)$, then
for any Borel set $A\subset\Om$,
\begin{equation}\label{eq:coarea}
\|Du\|(A)=\int_{-\infty}^{\infty}P(\{u>t\},A)\,dt.
\end{equation}

If $\Vert Du\Vert(\Om)<\infty$, from \eqref{eq:def of theta} and \eqref{eq:coarea}
we get the absolute continuity
\begin{equation}\label{eq:absolute continuity of var measure wrt H}
\Vert Du\Vert\ll \mathcal H\quad\textrm{on }\Om.
\end{equation}

If $u,v\in L^1_{\loc}(\Om)$, then
\begin{equation}\label{eq:variation of min and max}
\Vert D\min\{u,v\}\Vert(\Om)+\Vert D\max\{u,v\}\Vert(\Om)\le
\Vert Du\Vert(\Om)+\Vert Dv\Vert(\Om);
\end{equation}
for a proof see e.g. \cite[Lemma 3.1]{L-ZB}. Moreover, for any $u,v\in L^1_{\loc}(\Om)$, it is straightforward to show that
\begin{equation}\label{eq:BV functions form vector space}
\Vert D(u+v)\Vert(\Om)\le \Vert Du\Vert(\Om)+\Vert Dv\Vert(\Om).
\end{equation}

Since $\liploc(\Om)$ is dense in $N_{\loc}^{1,1}(\Om)$, see \cite[Theorem 5.47]{BB},
it follows that
\begin{equation}\label{eq:Sobolev subclass BV}
\Vert Du\Vert(\Om)\le \int_\Om g_u\,d\mu\quad
\textrm{for every }u\in N_{\loc}^{1,1}(\Om).
\end{equation}

The lower and upper approximate limits of a function $u$ on $\Om$
are defined respectively by
\[
u^{\wedge}(x):
=\sup\left\{t\in\R:\,\lim_{r\to 0}\frac{\mu(B(x,r)\cap\{u<t\})}{\mu(B(x,r))}=0\right\}
\]
and
\[
u^{\vee}(x):
=\inf\left\{t\in\R:\,\lim_{r\to 0}\frac{\mu(B(x,r)\cap\{u>t\})}{\mu(B(x,r))}=0\right\},
\]
for $x\in \Om$.
We then define the jump set as
\[
S_u:=\{u^{\wedge}<u^{\vee}\}.
\]
Note that since we understand $u^{\wedge}$ and $u^{\vee}$ to be defined only on $\Om$,
also $S_u$ is a subset of $\Om$.

Unlike Newton-Sobolev functions, we understand $\BV$ functions to be
$\mu$-equivalence classes. To consider fine properties, we need to
consider the pointwise representatives $u^{\wedge}$ and $u^{\vee}$.
The following fact clarifies the relationship between the different pointwise representatives;
it essentially follows from the Lebesgue point result for Newton-Sobolev functions given in
\cite{KKST2}.

\begin{proposition}[{\cite[Proposition 3.10]{L-ZB}}]\label{prop:Lebesgue points for Sobolev functions open set}
Let $\Om\subset X$ be an open set and let
$u\in N^{1,1}(\Om)$. Then $u=u^{\wedge}=u^{\vee}$ $\mathcal H$-a.e. in $\Om$.
\end{proposition}

By \cite[Theorem 5.3]{AMP}, the variation measure of a $\BV$ function
can be decomposed into the absolutely continuous and singular part, and the latter
into the Cantor and jump part
(which are all Radon measures), as follows. Given an open set 
$\Omega\subset X$ and $u\in\BV(\Omega)$, we have for any Borel set $A\subset \Om$
\begin{equation}\label{eq:variation measure decomposition}
\begin{split}
\Vert Du\Vert(A) &=\Vert Du\Vert^a(A)+\Vert Du\Vert^s(A)\\
&=\Vert Du\Vert^a(A)+\Vert Du\Vert^c(A)+\Vert Du\Vert^j(A)\\
&=\int_{A}a\,d\mu+\Vert Du\Vert^c(A)+\int_{A\cap S_u}\int_{u^{\wedge}(x)}^{u^{\vee}(x)}\theta_{\{u>t\}}(x)\,dt\,d\mathcal H(x),
\end{split}
\end{equation}
where $a\in L^1(\Omega)$ is the density of the absolutely continuous part
and the functions $\theta_{\{u>t\}}\in [\alpha,C_d]$ 
are as in~\eqref{eq:def of theta}.
It follows that $S_u$ is $\sigma$-finite with respect to $\mathcal H$.
Moreover, $\Vert Du\Vert^c(S)=0$ for any $S\subset \Om$ that is $\sigma$-finite with respect to $\mathcal H$.
If $\Vert Du\Vert^c(\Om)=0$, we say that $u\in\SBV(\Om)$.

\begin{definition}
	We say that a set $A\subset H$ is $1$-quasiopen with respect to a set $H\subset X$
	if for every $\eps>0$ there is an
	open set $G\subset X$ such that $\capa_1(G)<\eps$ and $A\cup G$ is relatively open
	in the subspace topology of $H$.
	
	We say that a set $A\subset H$ is $1$-quasiclosed with respect to $H$
	if $H\setminus A$ is $1$-quasiopen with respect to $H$, or
	equivalently, if for every $\eps>0$ there is an open set $G\subset X$
	such that $\capa_1(G)<\eps$ and $A\setminus G$ is relatively closed
	in the subspace topology of $H$.
	
	When $H=X$, we omit mention of it.
\end{definition}

Given $H\subset X$, we say that $u$ is $1$-quasi (lower/upper semi-)continuous on $H$ if
 for every $\eps>0$ there exists an open set $G\subset X$ such that $\capa_1(G)<\eps$
 and $u|_{H\setminus G}$ is real-valued (lower/upper semi-)continuous.

It is a well-known fact that Newton-Sobolev functions are quasicontinuous,
see \cite[Theorem 1.1]{BBS} or \cite[Theorem 5.29]{BB}.
This is also true in quasiopen sets; the following is a special case of
\cite[Theorem 1.3]{BBM-QO}. Note that $1$-quasiopen sets are $\mu$-measurable
by \cite[Lemma 9.3]{BB-OD}.

\begin{theorem}\label{thm:quasicontinuity}
Let $U\subset X$ be $1$-quasiopen and let $u\in N_{\loc}^{1,1}(U)$.
Then $u$ is $1$-quasicontinuous on $U$.
\end{theorem}

$\BV$ functions have the following partially analogous quasi-semicontinuity property.

\begin{proposition}\label{prop:quasisemicontinuity}
Let $\Om\subset X$ be open, let $u\in L^1_{\loc}(\Om)$ with
$\Vert Du\Vert(\Om)<\infty$, and let $\eps>0$. Then
$u^{\wedge}$ is $1$-quasi lower semicontinuous and
$u^{\vee}$ is $1$-quasi upper semicontinuous on $\Om$.
\end{proposition}
\begin{proof}
This follows from \cite[Corollary 4.2]{L-SA}, which is based on
\cite[Theorem 1.1]{LaSh}.
\end{proof}

We also have the following.

\begin{theorem}[{\cite[Theorem 4.3]{L-LSC}}]\label{thm:characterization of total variational}
	Let $U\subset X$ be $1$-quasiopen. If $\Vert Du\Vert(U)<\infty$, then
	\[
	\Vert Du\Vert(U)=\inf \left\{\liminf_{i\to\infty}\int_{U}g_{u_i}\,d\mu,\,
	u_i\in N_{\loc}^{1,1}(U),\, u_i\to u\textrm{ in }L^1_{\loc}(U)\right\},
	\]
	where each $g_{u_i}$ is the minimal $1$-weak upper gradient of $u_i$ in $U$.
\end{theorem}

For any $D\subset H\subset X$, with $H$ $\mu$-measurable, the space of Newton-Sobolev functions with zero boundary values is defined as
\[
N_0^{1,1}(D,H):=\{u|_{D}:\,u\in N^{1,1}(H)\textrm{ and }u=0\textrm { on }H\setminus D\}.
\]
The space is a subspace of $N^{1,1}(D)$ when $D$ is $\mu$-measurable, and it can always
be understood to be a subspace of $N^{1,1}(H)$.
If $H=X$, we omit it from the notation.

Similarly, for $D\subset\Om\subset X$, with $\Om$ open, we define the class of $\BV$ functions with zero boundary values as
\[
\BV_0(D,\Om):=
\left\{u|_{D}:\,u\in\BV(\Om),\ u^{\wedge}(x)=u^{\vee}(x)=0\textrm{ for }\mathcal H\textrm{-a.e. }x\in \Om\setminus D\right\}.
\]
This class was previously considered in \cite{L-ZB}.
Functions in $\BV_0(D,\Om)$ can also be understood to be defined on the whole of $\Om$,
and we will do so without further notice.
Moreover, if $\Om=X$, we omit it from the notation.
By \eqref{eq:Sobolev subclass BV}, Proposition
\ref{prop:Lebesgue points for Sobolev functions open set},
and \eqref{eq:null sets of Hausdorff measure and capacity} we see that
\begin{equation}\label{eq:Newtonian zero class contained in BV zero}
N_0^{1,1}(D,\Om)\subset \BV_0(D,\Om).
\end{equation}

Next we define the fine topology in the case $p=1$.
\begin{definition}\label{def:1 fine topology}
We say that $A\subset X$ is $1$-thin at the point $x\in X$ if
\[
\lim_{r\to 0}r\frac{\rcapa_1(A\cap B(x,r),B(x,2r))}{\mu(B(x,r))}=0.
\]
We also say that a set $U\subset X$ is $1$-finely open if $X\setminus U$ is $1$-thin at every $x\in U$. Then we define the $1$-fine topology as the collection of $1$-finely open sets on $X$.

We denote the $1$-fine interior of a set $H\subset X$, i.e. the largest $1$-finely open set contained in $H$, by $\fint H$. We denote the $1$-fine closure of $H$,
i.e. the smallest $1$-finely closed set containing $H$, by $\overline{H}^1$. The $1$-fine boundary of $H$
is $\partial^1 H:=\overline{H}^1\setminus \fint H$.
\end{definition}

See \cite[Section 4]{L-FC} for discussion on this definition, and for a proof of the fact that the
$1$-fine topology is indeed a topology.
By \cite[Lemma 3.1]{L-Fed}, $1$-thinness implies zero measure density, i.e.
\begin{equation}\label{eq:thinness and measure thinness}
\textrm{If }A\textrm{ is 1-thin at }x,\textrm{ then }x\in O_A.
\end{equation}

\begin{theorem}[{\cite[Corollary 6.12]{L-CK}}]\label{thm:finely open is quasiopen and vice versa}
A set $U\subset X$ is $1$-quasiopen if and only if it is the union of a $1$-finely
open set and a $\mathcal H$-negligible set.
\end{theorem}

\begin{lemma}[{\cite[Lemma 4.9]{L-LSC}}]\label{lem:stability of quasiopen sets}
	Let $U\subset X$ be $1$-quasiopen and let $A\subset X$ be $\mathcal H$-negligible.
	Then $U\setminus A$ and $U\cup A$ are $1$-quasiopen sets.
\end{lemma}

\section{$\BV$ functions with zero boundary values}\label{sec:BV0 class}

In this section we consider some questions related to the
class $\BV_0(D,\Om)$, which
will be needed in later sections.
We will always denote by $\Om$ a nonempty open set.

The support of a function $u$ defined on a subset of
$\Om$ (usually the entire $\Om$, except in Lemma
\ref{lem:extension of compactly supported function} below)
is the relatively closed
(in the subspace topology of $\Om$) set
\[
\supp_{\Om} u:=\{x\in \Om:\,\mu(B(x,r)\cap\{u\neq 0\})>0\ \textrm{for all }r>0\}.
\]

\begin{theorem}[{\cite[Theorem 3.16]{L-ZB}}]\label{thm:characterization of BV function with zero bdry values}
	Let $D \subset \Om\subset X$,
	and let $u\in\BV(\Om)$.
	Then the following are equivalent:
	\begin{enumerate}[{(1)}]
		\item $u\in \BV_0(D,\Om)$.
		\item There exists a sequence $(u_k)\subset\BV(\Om)$ such that each $\supp_{\Om} u_k$ is
		a bounded subset of $D$, and
		$u_k\to u$ in $\BV(\Om)$.
	\end{enumerate}
\end{theorem}

The following lemma, though slightly technical, simply shows that we can apply the
definition of the total variation to find approximating locally Lipschitz functions that converge suitably in the $L^1$-norm.

\begin{lemma}\label{lem:choosing approximating function}
	Let $\Om_1\Subset \Om_2\Subset \ldots \Subset\bigcup_{j=1}^{\infty}\Om_j=\Om$ be open sets, let $\Om_0:=\emptyset$, and
	let $\eta_j\in \Lip_c(\Om_{j+1})$ such that $0\le \eta_j\le 1$ on $X$ and
	$\eta_j=1$ on $\Om_{j}$ for each $j\in\N$, and $\eta_0\equiv 0$.
	Moreover, let $u\in L^1_{\loc}(\Om)$ with
	$\Vert Du\Vert(\Om)<\infty$, and 
	let $(u_i)\subset \liploc(\Om)$ such that $u_i\to u$ in $L^1_{\loc}(\Om)$
	and
	\[
	\lim_{i\to\infty}\int_\Om g_{u_i}\,d\mu=\Vert Du\Vert(\Om),
	\]
	where each $g_{u_i}$ is the minimal $1$-weak upper gradient of $u_i$ in $\Om$.
	Finally, let $\delta_j>0$ for each $j\in\N$, and let $\eps>0$. Then,
	passing to a suitable subsequence of $(u_i)$ (not relabeled,
	and with the understanding that terms can be repeated) and defining
	\begin{equation}\label{eq:definition by means of etas}
	v:=\sum_{i=1}^{\infty}(\eta_i-\eta_{i-1})u_{i},
	\end{equation}
	we have $\Vert v-u\Vert_{L^1(\Om_{j}\setminus \Om_{j-1})}<\delta_j$ for all $j\in\N$
	and
	$\int_\Om g_v\,d\mu<\Vert Du\Vert(\Om)+\eps$.
\end{lemma}

\begin{proof}
	By the definition of the total variation,
	we have $\Vert Du\Vert(W)\le\liminf_{i\to\infty}\int_W g_{u_i}\,d\mu$
	for any open $W\subset \Om$, and thus $\Vert Du\Vert(F)\ge\limsup_{i\to\infty}\int_F g_{u_i}\,d\mu$
	for any closed $F\subset \Om$, and so
	\[
	\limsup_{i\to\infty}\int_{\Om_{j+1}\setminus \Om_{j-1}}g_{u_i}\,d\mu\le
	\Vert Du\Vert(\overline{\Om}_{j+1}\setminus \Om_{j-1})
	\]
	for each $j\in\N$.
	Denote by $L_j>0$ (some) Lipschitz constants of the functions $\eta_j$;
	we can take this to be an increasing sequence.
	By passing to a subsequence of $(u_i)$ (not relabeled), we can assume that
	\begin{equation}\label{eq:choosing the L1 closeness of uis}
	\Vert u_{i-1}-u\Vert_{L^1(\Om_{i})}
	<\min\{\delta_{i-1},\delta_{i},2^{-i+1}\eps/L_{i-1}\}/2
	\end{equation}
	and
	\begin{equation}\label{eq:choosing the energy of uis}
	\int_{\Om}g_{u_1}\,d\mu\le
	\Vert Du\Vert(\Om)+\eps,\quad
	\int_{\Om_{i+1}\setminus \Om_{i-1}}g_{u_i}\,d\mu\le
	\Vert Du\Vert(\overline{\Om}_{i+1}\setminus \Om_{i-1})+2^{-i}\eps
	\end{equation}
	for all $i=2,3\ldots$.
	We can also assume that for $k\in\N$ to be chosen later, $u_1=\ldots =u_k$.
	We have
	\begin{align*}
	\Vert v-u\Vert_{L^1(\Om_{j}\setminus \Om_{j-1})}
	&=\Vert \sum_{i=1}^{\infty}(\eta_i-\eta_{i-1})u_{i}-u\Vert_{L^1(\Om_{j}\setminus \Om_{j-1})}\\
	&=\Vert \eta_{j-1}u_{j-1}+(1-\eta_{j-1})u_{j}-u\Vert_{L^1(\Om_{j}\setminus \Om_{j-1})}\\
	&\le \Vert u_{j-1}-u\Vert_{L^1(\Om_{j}\setminus \Om_{j-1})}
	+\Vert u_{j}-u\Vert_{L^1(\Om_{j}\setminus \Om_{j-1})}\\
	&<\delta_j
	\end{align*}
	as desired.
	Let $v_1:=u_1$ and recursively $v_{i+1}:=\eta_{i} v_i+(1-\eta_{i})u_{i+1}$.
	We see that $v=\lim_{i\to\infty}v_i$.
	By a Leibniz rule \cite[Lemma 2.18]{BB}, the minimal $1$-weak upper gradient
	of $v_2$ in $\Om$ satisfies
	\[
	g_{v_2}\le g_{\eta_1}|u_1-u_2|+\eta_1 g_{u_1}+(1-\eta_1)g_{u_2}.
	\]
	Inductively, we get
	\[
	g_{v_i}\le \sum_{j=1}^{i-1} g_{\eta_j}|u_j-u_{j+1}|+
	\sum_{j=1}^{i-1}(\eta_{j}-\eta_{j-1}) g_{u_j}+(1-\eta_{i-1})g_{u_{i}};
	\]
	to prove this, assume that it holds for the index $i$. Then we have by applying a Leibniz rule as above, and noting that $g_{\eta_i}$ can be nonzero only in $\Om_{i+1}\setminus \Om_i$ (see \cite[Corollary 2.21]{BB}), where $v_i=u_i$,
	\begin{align*}
	g_{v_{i+1}}
	&\le g_{\eta_{i}}|v_i-u_{i+1}|+\eta_{i}
	g_{v_i}+(1-\eta_{i})g_{u_{i+1}}\\
	&\le g_{\eta_{i}}|u_{i}-u_{i+1}|+\sum_{j=1}^{i-1} g_{\eta_j}|u_j-u_{j+1}|\\
	&\qquad
	+\sum_{j=1}^{i-1}(\eta_j-\eta_{j-1})g_{u_{j}}+(\eta_{i}-\eta_{i-1})g_{u_{i}}
	+(1-\eta_{i})g_{u_{i+1}}.
	\end{align*}
	This completes the induction. Thus in each $\Om_i$, where $v=v_{i+1}$,
	the minimal $1$-weak upper gradient of $v$ in $\Om_i$ satisfies
	\begin{align*}
	g_{v}=g_{v_{i+1}}
	&\le \sum_{j=1}^{\infty} g_{\eta_j}|u_j-u_{j+1}|+
	\sum_{j=1}^{\infty}(\eta_{j}-\eta_{j-1}) g_{u_j}\\
	&= \sum_{j=1}^{\infty} g_{\eta_j}|u_j-u_{j+1}|+\eta_{k}g_{u_1}+
	\sum_{j=k+1}^{\infty}(\eta_{j}-\eta_{j-1}) g_{u_j},
	\end{align*}
	since $u_1=\ldots =u_k$. Thus
	\begin{align*}
	&\int_{\Om_i} g_v\,d\mu\\
	&\qquad\le \sum_{j=1}^{\infty}\int_{\Om} g_{\eta_j}|u_j-u_{j+1}|\,d\mu+
	\int_{\Om}\eta_{k}g_{u_1}\,d\mu+\sum_{j=k+1}^{\infty}\int_{\Om} (\eta_{j}-\eta_{j-1}) g_{u_j}\,d\mu\\
	&\qquad\overset{\eqref{eq:choosing the energy of uis}}{\le} \sum_{j=1}^{\infty}L_j\Vert u_j-u_{j+1}\Vert_{L^1(\Om_{j+1}\setminus \Om_j)}
	+\Vert Du\Vert(\Om)+\eps+\sum_{j=k+1}^{\infty}\int_{\Om_{j+1}\setminus \Om_{j-1}} g_{u_j}\,d\mu\\
	&\qquad\overset{\eqref{eq:choosing the L1 closeness of uis},\eqref{eq:choosing the energy of uis}}{\le} 3\eps+\Vert Du\Vert(\Om)
	+\sum_{j=k+1}^{\infty}\Vert Du\Vert(\overline{\Om}_{j+1}\setminus \Om_{j-1})\\
	&\qquad\le 3\eps
	+\Vert Du\Vert(\Om)+3\Vert Du\Vert(\Om\setminus \Om_{k})\\
	&\qquad\le \Vert Du\Vert(\Om)+4\eps,
	\end{align*}
	if we choose $k$ large enough.
	Note that $g_v$ does not depend on $i$, see
	\cite[Lemma 2.23]{BB}, and so it is well
	defined on $\Om$.
	Since $g_v$ is the minimal $1$-weak upper gradient of $v$ in
	each $\Om_i$, it is clearly
	also (the minimal) $1$-weak upper gradient of $v$ in $\Om$.
	Then by Lebesgue's monotone convergence theorem,
	\[
	\int_{\Om} g_v\,d\mu\le \Vert Du\Vert(\Om)+4\eps.
	\]
\end{proof}

In a rather similar way, we prove the following lemma which we will need later.

\begin{lemma}\label{lem:BV pasting lemma}
	Let $u\in L^1_{\loc}(\Om)$ with $\Vert Du\Vert(\Om)<\infty$
	and let $(u_i)\subset L^1_{\loc}(\Om)$ such that
	$\Vert u_i- u\Vert_{L^{\infty}(\Om)}\to 0$
	and
	\[
	\lim_{i\to\infty}\Vert Du_i\Vert(\Om)=\Vert Du\Vert(\Om).
	\]
	Let $\eps>0$. Then we find a function
	$v\ge u$ such that $\Vert v-u\Vert_{L^1(\Om)}<\eps$, $\Vert v-u\Vert_{L^{\infty}(\Om)}<\eps$,
	$\Vert Dv\Vert(\Om)<\Vert Du\Vert(\Om)+\eps$,
	\[
	\lim_{\Om\ni y\to x}|v-u|^{\vee}(y)=0\quad\textrm{for all }x\in \partial\Om,
	\]
 	and $S_v \subset \bigcup_{i=1}^{\infty}S_{u_i}$.
	Moreover, $\Vert Dv\Vert^c(\Om)=0$ if $\Vert Du_i\Vert^c(\Om)=0$ for all
	$i\in\N$.
\end{lemma}

\begin{proof}
	Take nonempty open sets
	$\Om_1\Subset \Om_2\Subset \ldots \Subset\bigcup_{j=1}^{\infty}\Om_j=\Om$,
	and $\Om_0:=\emptyset$. Also take
	functions $\eta_j\in \Lip_c(\Om_{j+1})$ such that $0\le \eta_j\le 1$ on $X$ and
	$\eta_j=1$ on $\Om_{j}$ for each $j\in\N$, and $\eta_0\equiv 0$.
	
	By replacing the functions $u_i$ with
	$u_i+\Vert u_i-u\Vert_{L^{\infty}(\Om)}$, we can assume that
	$u_i\ge u$ on $\Om$ for each $i\in\N$.
	By passing to a subsequence (not relabeled),
	we can assume that for each $i\in\N$,
	\begin{equation}\label{eq:choice of uis and uniform convergence}
	\Vert u_i-u\Vert_{L^{\infty}(\Om)}<2^{-i}\eps\min\left\{1,\mu(\{\eta_i>0\})^{-1},\,
	\int_{\Om}g_{\eta_{i-1}}\,d\mu,
	\int_{\Om}g_{\eta_i}\,d\mu\right\}.
	\end{equation}
	From the fact that $\lim_{i\to\infty}\Vert Du_i\Vert(\Om)=\Vert Du\Vert(\Om)$
	and from the lower semicontinuity of the total variation in open sets,
	it follows that for each $j\in\N$ (see \cite[Proposition 1.80]{AFP})
	\[
	\lim_{i\to\infty}\int_{\Om} (1-\eta_{j-1}) \,d\Vert Du_i\Vert
	=\int_{\Om} (1-\eta_{j-1})\,d\Vert Du\Vert
	\]
	and
	\[
	\lim_{i\to\infty}\int_{\Om} (\eta_{j}-\eta_{j-1}) \,d\Vert Du_i\Vert
	=\int_{\Om} (\eta_{j}-\eta_{j-1})\,d\Vert Du\Vert.
	\]
	Thus we can also assume that for each $i\in\N$,
	\begin{equation}\label{eq:choice of uis and weak convergence 1 minus eta}
	\int_{\Om} (1-\eta_{i-1}) \,d\Vert Du_i\Vert
	<\int_{\Om} (1-\eta_{i-1})\,d\Vert Du\Vert+2^{-i}\eps
	\end{equation}
	and
	\begin{equation}\label{eq:choice of uis and weak convergence}
	\int_{\Om} (\eta_{i}-\eta_{i-1}) \,d\Vert Du_i\Vert
	<\int_{\Om} (\eta_{i}-\eta_{i-1})\,d\Vert Du\Vert+2^{-i}\eps.
	\end{equation}
	Let
	\begin{equation}\label{eq:definition of w by means of etas}
	v:=\sum_{j=1}^{\infty}(\eta_{j}-\eta_{j-1})u_{j}.
	\end{equation}
	Then $v\ge u$ and
	\begin{align*}
	\Vert v-u\Vert_{L^1(\Om)}
	=\Vert \sum_{j=1}^{\infty}(\eta_j-\eta_{j-1})(u_{j}-u)\Vert_{L^1(\Om)}
	\le \sum_{j=1}^{\infty} \mu(\{\eta_j>0\})\Vert u_j-u\Vert_{L^{\infty}(\Om)}
	<\eps
	\end{align*}
	by \eqref{eq:choice of uis and uniform convergence}.
	Clearly also $\Vert v-u\Vert_{L^{\infty}(\Om)}<\eps$.
	
	Let $v_1:=u_1$
	and then recursively $v_{i+1}:=\eta_{i} v_i+(1-\eta_{i})u_{i+1}$.
	Then $v=\lim_{i\to\infty}v_i$.
	By induction, we show that in $\Om$,
	\begin{equation}\label{eq:formula for Dvi}
	d\Vert Dv_i\Vert
	\le \sum_{j=1}^{i-1} g_{\eta_j}|u_j-u_{j+1}|\,d\mu
	+\sum_{j=1}^{i-1}(\eta_{j}-\eta_{j-1}) d\Vert Du_j\Vert+(1-\eta_{i-1})\,d\Vert Du_{i}\Vert;
	\end{equation}
	this clearly holds for $i=1$.
	Note that $g_{\eta_i}$ can be nonzero only in $\Om_{i+1}\setminus \Om_i$
	(see \cite[Corollary 2.21]{BB}),
	where $v_i=u_i$.
	Assuming that \eqref{eq:formula for Dvi} holds for the index $i$,
	by a Leibniz rule (see \cite[Lemma 3.2]{HKLS}) we get
	\begin{align*}
	d\Vert Dv_{i+1}\Vert
	&\le g_{\eta_{i}}|v_i-u_{i+1}|\,d\mu+\eta_{i}\,
	d\Vert Dv_{i}\Vert+(1-\eta_{i})\,d\Vert Du_{i+1}\Vert\\
	&\le g_{\eta_{i}}|u_{i}-u_{i+1}|\,d\mu
	+\sum_{j=1}^{i-1} g_{\eta_j}|u_j-u_{j+1}|\,d\mu
	+\sum_{j=1}^{i-1}(\eta_{j}-\eta_{j-1}) d\Vert Du_j\Vert\\
	&\qquad \qquad+(\eta_{i}-\eta_{i-1})\,d\Vert Du_{i}\Vert+(1-\eta_{i})\,d\Vert Du_{i+1}\Vert.
	\end{align*}
	This completes the induction.
	Thus since $v_i\to v$ in $L^1_{\loc}(\Om)$, we get
	\begin{align*}
	&\Vert Dv\Vert(\Om)\le\liminf_{i\to\infty}\Vert Dv_i\Vert(\Om)\\
	&\le \sum_{j=1}^{\infty}\int_{\Om} g_{\eta_j}|u_j-u_{j+1}|\,d\mu+
	\sum_{j=1}^{\infty}\int_{\Om} (\eta_{j}-\eta_{j-1}) \,d\Vert Du_j\Vert
	+ \liminf_{i\to\infty}(1-\eta_{i-1})\,d\Vert Du_{i}\Vert\\
	&< 2\sum_{j=1}^{\infty}2^{-j}\eps
	+\sum_{j=1}^{\infty}\left(\int_{\Om} (\eta_{j}-\eta_{j-1})\,d\Vert Du\Vert
	+2^{-j}\eps\right)
	\qquad\textrm{by }\eqref{eq:choice of uis and uniform convergence},
	\eqref{eq:choice of uis and weak convergence},
	\eqref{eq:choice of uis and weak convergence 1 minus eta}\\
	&= 3\eps+\Vert Du\Vert(\Om),
	\end{align*}
	as desired.
	Next, note that \eqref{eq:definition of w by means of etas}
	is a locally finite sum.
	If $x\notin S_{u_j}$, clearly
	$x\notin S_{(\eta_{j}-\eta_{j-1})u_{j}}$. Thus
	$S_v\subset \bigcup_{j=1}^{\infty}S_{u_j}$.
	
	If $\Vert Du_i\Vert^c(\Om)=0$ for all $i\in\N$, we show that
	$\Vert Dv\Vert^c(\Om)=0$
	as follows. Let $F\subset \Om$ be a $\mu$-negligible set such that
	$\Vert Dv\Vert^c(\Om\setminus F)=0$.
	Note that $\Vert Dv\Vert=\Vert Dv_{i+1}\Vert$
	in $\Om_i$, and so by \eqref{eq:formula for Dvi},
	in $\Om_i$ we have
	\begin{equation}\label{eq:estimate for Dv}
	d\Vert Dv\Vert=d\Vert Dv_{i+1}\Vert
	\le \sum_{j=1}^{\infty} g_{\eta_j}|u_j-u_{j+1}|\,d\mu
	+\sum_{j=1}^{\infty}(\eta_{j}-\eta_{j-1}) d\Vert Du_j\Vert.
	\end{equation}
	Since this inequality holds in every $\Om_i$,
	it holds in $\Om$.
	By the discussion after \eqref{eq:variation measure decomposition},
	\begin{align*}
	\Vert Dv\Vert^c(F)
	&=\Vert Dv\Vert^c\left(F\setminus \bigcup_{i=1}^{\infty}S_{u_i}\right)\\
	&\le \sum_{j=1}^{\infty}\Vert Du_j\Vert\left(F\setminus\bigcup_{i=1}^{\infty}S_{u_i}\right)\quad\textrm{by }\eqref{eq:estimate for Dv}\\
	&=0
	\end{align*}
	since $\Vert Du_j\Vert^a(F)=0$ and $\Vert Du_j\Vert^j(\Om\setminus S_{u_j})=0$. Thus $\Vert Dv\Vert^c(\Om)=0$.
\end{proof}

The next simple lemma shows the existence of suitable cutoff functions.

\begin{lemma}\label{lem:cutoff function lemma}
	Let $W\subset\Om\subset X$ be open sets and let $H\subset W$ be relatively closed
	(in the subspace topology of $\Om$). Then there is
	a function $\eta\in \liploc(\Om)$ such that $0\le \eta\le 1$,
	$\eta=1$ on $H$, and $\supp_{\Om} \eta\subset W$.
	
	Moreover, if $H$ is bounded, also $\supp_{\Om}\eta$ is bounded.
\end{lemma}
\begin{proof}
Take open sets $\Om_1\Subset \Om_2\Subset \ldots \Subset\bigcup_{j=1}^{\infty}\Om_j=\Om$
and $\Om_0:=\emptyset$.
	Note that for each $j=0,1,\ldots$, $H\cap \overline{\Om}_{j+1}\setminus \Om_{j}$ is a compact subset of the open set
	$W\cap \Om_{j+2}\setminus \overline{\Om_{j-1}}$. Take
	$\eta_j\in \Lip_c(W\cap \Om_{j+2}\setminus \overline{\Om_{j-1}})$ such that $0\le \eta_j\le 1$ and $\eta_j=1$ on
	$H\cap \overline{\Om}_{j+1}\setminus \Om_{j}$. Let
	\[
	\eta:=\sup_{j\in\N}\eta_j.
	\]
	Now it is straightforward to check that $\eta$ has the required
	properties.
	If $H$ is bounded, we can also choose the $\eta_j$'s so that $\eta_j=0$ outside
	a $1$-neighborhood of $H$, ensuring that $\supp_{\Om}\eta$ is bounded.
\end{proof}

\begin{lemma}\label{lem:extension of compactly supported function}
	Let $W\subset\Om\subset X$ be open sets and let $u\in\BV(W)$
	such that $\supp_{\Om} u\subset W$.
	Then there exists a sequence $(u_i)\subset\liploc(W)$
	such that $\supp_{\Om} u_i$ are subsets of $W$ and bounded if
	$\supp_{\Om} u$ is,
	$u_i\to u$ in $L^1(W)$, and
	\begin{equation}\label{eq:total variation in W}
	\Vert Du\Vert(W)=\lim_{i\to\infty}\int_{W}g_{u_i}\,d\mu,
	\end{equation}
	where each $g_{u_i}$ is the minimal $1$-weak upper gradient of $u_i$ in $W$.
	Moreover, $u\in\BV_0(W,\Om)$ with $\Vert Du\Vert(W)=\Vert Du\Vert(\Om)$
	(by zero extension to $\Om\setminus W$),
	and then \eqref{eq:total variation in W}
	holds also with $W$ replaced by $\Om$.
\end{lemma}

\begin{proof}
	Fix $\eps>0$.
	By Lemma \ref{lem:cutoff function lemma} we find a function $\eta\in \liploc(\Om)$ such that $0\le \eta\le 1$,
	$\eta=1$ on $\supp_{\Om} u$, and $\supp_{\Om} \eta$ is a subset of $W$
	and bounded if $\supp_{\Om} u$ is.
	Take open sets $W_1\Subset W_2\Subset \ldots \Subset \bigcup_{j=1}^{\infty}W_j=W$ and $W_0:=\emptyset$.
	Denote by $L_j>0$ the Lipschitz constant of $\eta$ in $W_{j}$.
	By Lemma \ref{lem:choosing approximating function} we find a function
	$v\in \liploc(W)$ such that $\int_W g_v\,d\mu\le \Vert Du\Vert(W)+\eps$ and
	$\Vert v-u\Vert_{L^1(W_{j}\setminus W_{j-1})}<2^{-j}\eps\min\{1,L_j^{-1}\}$ for every $j\in\N$.
	Then
	\[
	\Vert \eta v-u\Vert_{L^1(W)}=\Vert \eta v-\eta u\Vert_{L^1(W)}
	\le \Vert v-u\Vert_{L^1(W)}<\eps.
	\]
	Moreover, $g_{\eta}=0$ on $\supp_{\Om} u$ by \cite[Corollary 2.21]{BB}, and then by the Leibniz rule \cite[Theorem 2.15]{BB},
	\begin{align*}
	\int_W g_{\eta v}\,d\mu
	&\le \int_{W}g_{v}\,d\mu+g_{\eta}|v|\,d\mu\\
	&= \int_{W}g_{v}\,d\mu+g_{\eta}|v-u|\,d\mu\\
	&= \int_{W}g_{v}\,d\mu+\sum_{j=1}^{\infty}\int_{W_j\setminus W_{j-1}}g_{\eta}|v-u|\,d\mu\\
	&\le \int_{W}g_{v}\,d\mu+\sum_{j=1}^{\infty}L_j\Vert v-u\Vert_{L^1(W_{j}\setminus W_{j-1})}\\
	&\le \Vert Du\Vert(W)+2\eps.
	\end{align*}
	We find the desired functions by letting $u_i:=\eta v$ with the choices $\eps=1/i$.
	To prove the second claim,
	denote by $u,u_i$ also the zero extensions of these functions
	to $\Om\setminus W$.
	Obviously $u_i\to u$ in $L^1(\Om)$.
	Note that the minimal $1$-weak upper gradient $g_{u_i}$
	(now as a function defined on $\Om$)
	is clearly the zero extension of $g_{u_i}$ (as a function defined only on $W$),
	and so we have
	\[
	\Vert Du\Vert(\Om)\le \liminf_{i\to\infty}\int_\Om g_{u_i}\,d\mu=\liminf_{i\to\infty}\int_{W} g_{u_i}\,d\mu=\Vert Du\Vert(W).
	\]
	Thus $u\in\BV(\Om)$ and then clearly $u\in\BV_0(W,\Om)$.
\end{proof}

Now we can show that Lipschitz functions with zero boundary values
are dense in the class $\BV_0(W,\Om)$ in the following weak sense. 

\begin{proposition}\label{prop:weak density of lipschitz functions}
	Let $W\subset \Om\subset X$ be open sets and let $u\in\BV_0(W,\Om)$. Then there exists a sequence
	$(u_i)\subset \liploc(\Om)$ such that each $\supp_{\Om}u_i\subset W$ is bounded,
	$u_i\to u$ in $L^1(\Om)$, and
	\[
	\lim_{i\to\infty}\int_{\Om}g_{u_i}\,d\mu=\Vert Du\Vert(\Om).
	\]
\end{proposition}

\begin{proof}
	By Theorem \ref{thm:characterization of BV function with zero bdry values},
	we find a sequence $(v_i)\subset\BV(\Om)$ such that
	$\supp_{\Om} v_i\subset W$ are bounded and
	$\Vert v_i-u\Vert_{\BV(\Om)}<1/i$ for each $i\in\N$.
	Then by Lemma \ref{lem:extension of compactly supported function}, for each $i\in\N$ we find $u_i\in\liploc(\Om)$ such that
	$\supp_{\Om} u_i\subset W$ is
	bounded,
	$\Vert u_i-v_i\Vert_{L^1(\Om)}<1/i$, and
	\[
	\left|\int_{\Om} g_{u_i}\,d\mu- \Vert Dv_i\Vert(\Om)\right|< 1/i.
	\]
	We conclude that $\Vert u_i-u\Vert_{L^1(\Om)}<2/i$ and
	\[
	\left|\int_{\Om} g_{u_i}\,d\mu- \Vert Du\Vert(\Om)\right|< 2/i.
	\]
\end{proof}

The analog of the following result is well
known for Newton-Sobolev functions, see \cite[Lemma 2.37]{BB},
and so it is natural to prove it here for the class $\BV_0(W,\Om)$,
even though we will not need this result later.

\begin{proposition}\label{prop:uvw}
	Let $W\subset \Om\subset X$ be open sets and let  $u\in \BV(W)$ and $v,w\in \BV_0(W,\Om)$ such that
	$v\le u\le w$ in $\Om$.
	Then $u\in\BV_0(W,\Om)$.
\end{proposition}

\begin{proof}
	By observing that $u\in \BV_0(W,\Om)$ if and only if $u-v\in \BV_0(W,\Om)$,
	we can assume that $v\equiv 0$.
	Denote the zero extension of $u$ to $\Om\setminus W$ by $u_0$. 
	By Theorem \ref{thm:characterization of BV function with zero bdry values},
	we find a sequence of
	nonnegative functions $(w_k)\subset\BV(\Om)$ with $\supp_{\Om}w_k\subset W$
	and $w_k\to w$ in $\BV(\Om)$ (the nonnegativity can be achieved by truncating, if needed).
	Then $\varphi_k:=\min\{w_k,u_0\}\in\BV(W)$ by
	\eqref{eq:variation of min and max} and $\varphi_k\in \BV_0(W,\Om)$ by Lemma
	\ref{lem:extension of compactly supported function}, for each $k\in\N$.
	Moreover, $\varphi_k\to u_0$ in $L^1(\Om)$ and
	\begin{align*}
	\liminf_{k\to\infty}\Vert D\varphi_k\Vert(\Om)
	&=\liminf_{k\to\infty}\Vert D\varphi_k\Vert(W)\quad\textrm{by Lemma }
	\ref{lem:extension of compactly supported function}\\
	&\le \liminf_{k\to\infty}\Vert Dw_k\Vert(W)
	+\Vert Du_0\Vert(W)\quad\textrm{by }\eqref{eq:variation of min and max}\\
	&= \Vert Dw\Vert(W)+\Vert Du\Vert(W).
	\end{align*}
	Thus by the lower semicontinuity of the total variation with respect to $L^1$-convergence, $u_0\in\BV(\Om)$. Moreover,
	$u_0^{\vee}(x)\le w^{\vee}(x)=0$ for $\mathcal H$-a.e. $x\in \Om\setminus W$,
	and obviously $u_0^{\wedge}(x)\ge 0$ for all $x\in \Om\setminus W$,
	guaranteeing that $u_0^{\wedge}=u_0^{\vee}=0$ $\mathcal H$-a.e. in $\Om\setminus W$.
\end{proof}

\section{Variational capacities}\label{sec:capacities}

In this section we study variational capacities.
Our approximation result will be based on the main result of this section,
Theorem \ref{thm:1capacity and BVcapacity are equal}.

We begin by defining the variational $1$-capacity and its Lipschitz and $\BV$ analogs.

\begin{definition}\label{def:variational capacities}
	Let $A\subset D\subset H\subset X$ be
	nonempty sets such that $H$ is $\mu$-measurable.
	We define the variational (Newton-Sobolev) $1$-capacity by
	\[
	\rcapa_1(A,D,H):=\inf \int_H g_u\,d\mu,
	\]
	where the infimum is taken over functions $u\in N^{1,1}_0(D,H)$ such that
	$u\ge 1$ on $A$.
	
	We define the variational Lipschitz $1$-capacity by
	\[
	\rcapa_{\mathrm{lip},1}(A,D,H):=\inf \int_{H}g_u\,d\mu,
	\]
	where the infimum is taken over functions 
	$u\in N^{1,1}_0(D,H)\cap \liploc(H)$
	such that $u\ge 1$ on $A$.
	
	Finally, we define the variational $\BV$-capacity by
	\[
	\rcapa_{\BV}(A,D,H):=\inf \Vert Du\Vert(H),
	\]
	where the infimum is taken over functions $u\in L^1(H)$ such that
	$u^{\wedge}=u^{\vee}= 0$ $\mathcal H$-a.e. on $H\setminus D$ and
	$u^{\wedge}\ge 1$ $\mathcal H$-a.e. on $A$.
	
	If $H=X$, we omit it from the notation.
	In each case, we say that the functions $u$ over which we take the infimum are admissible (test) functions
	for the capacity in question.
\end{definition}

Again, $g_u$ always denotes the minimal $1$-weak upper gradient of $u$ (in $H$).
Recall that we understand Newton-Sobolev functions to be defined at every point, but in the definition of $\rcapa_1(A,D,H)$ we can equivalently
require $u\ge 1$ $1$-q.e. on $A$,
by \eqref{eq:quasieverywhere equivalence classes}.
However, the same is not true for $\rcapa_{\mathrm{lip},1}(A,D,H)$.
In each definition, we see by truncation that it is enough to consider test functions $0\le u\le 1$,
and then the conditions $u\ge 1$ and $u^{\wedge}\ge 1$ are replaced by $u= 1$ and $u^{\wedge}= 1$, respectively.

In the definition of the variational $\BV$-capacity, it is implicitly understood
that the test functions need to satisfy $\Vert Du\Vert(\Om)<\infty$ for some open
$\Om\supset H$.
Note that if $H$ is itself open, then the infimum is taken over functions
$u\in\BV_0(D,H)$ such that
$u^{\wedge}\ge 1$ $\mathcal H$-a.e. on $A$.

Using \eqref{eq:Sobolev subclass BV},
\eqref{eq:Newtonian zero class contained in BV zero},
and Proposition \ref{prop:Lebesgue points for Sobolev functions open set},
it is straightforward to see that for open $\Om\subset X$,
\begin{equation}\label{eq:general capacity inequalities}
\rcapa_{\BV}(A,D,\Om)\le\rcapa_{1}(A,D,\Om)\le\rcapa_{\lip,1}(A,D,\Om).
\end{equation}
In \cite[Theorem 4.23]{L-ZB} it was shown
that for a compact subset $A$ of an open set $D$, we have
\[
\rcapa_{\BV}(A,D)=\rcapa_{1}(A,D)=\rcapa_{\lip,1}(A,D).
\]
For more general $A,D$, if there is no positive distance between $A$ and
$X\setminus D$, then of course $\rcapa_{\lip,1}(A,D)=\infty$, while the other
capacities may be finite. However, it is natural to expect the
equality $\rcapa_{\BV}(A,D)=\rcapa_{1}(A,D)$ to hold more generally.
We show this in Theorem \ref{thm:1capacity and BVcapacity are equal} below.

\begin{lemma}[{\cite[Lemma 3.4]{L-LSC}}]\label{lem:capacity and Newtonian function and bigger set}
	Let $G\subset X$ and $\eps>0$. Then there exists an open set $V\supset G$ with $\capa_1(V)\le C_1(\capa_1(G)+\eps)$ and a
	function $\eta\in N^{1,1}_0(V)$ with $0\le\eta\le 1$ on $X$, $\eta=1$ on $G$, and $\Vert \eta\Vert_{N^{1,1}(X)}\le C_1(\capa_1(G)+\eps)$,
	for some constant $C_1=C_1(C_d,C_P,\lambda)\ge 1$.
\end{lemma}

\begin{lemma}[{\cite[Lemma 3.8]{L-SA}}]\label{lem:variation measure and capacity}
	Let $\Omega\subset X$ be an open set and
	let $u\in L^1_{\loc}(\Omega)$ with $\Vert Du\Vert(\Omega)<\infty$. Then for every
	$\eps>0$ there exists $\delta>0$ such that if $A\subset \Omega$ with $\capa_1 (A)<\delta$,
	then $\Vert Du\Vert(A)<\eps$.
\end{lemma}

\begin{theorem}\label{thm:1capacity and BVcapacity are equal}
Let $D\subset \Om'\subset X$ be $1$-quasiopen sets and let
$A\subset D$ be $1$-quasiclosed with respect to $\Om'$.
Then $\rcapa_1(A,D,\Om')\le \rcapa_{\BV}(A,D,\Om')$.
\end{theorem}

Note that if $\Om'$ is in fact open, then by \eqref{eq:general capacity inequalities}
we have $\rcapa_1(A,D,\Om')= \rcapa_{\BV}(A,D,\Om')$.

\begin{proof}
We can assume that $\rcapa_{\BV}(A,D,\Om')<\infty$. Fix $\eps>0$.
Take an open set $\Om\supset \Om'$ and a function $h\in L^1(\Om')$ such that
$0\le h\le 1$ on $\Om$, $h^{\vee}=0$ $\mathcal H$-a.e. on $\Om'\setminus D$,
$h^{\wedge}=1$ $\mathcal H$-a.e. on $A$, and
$\Vert Dh\Vert(\Om)<\rcapa_{\BV}(A,D,\Om')+\eps$.
As $\Om'$ is $1$-quasiopen we can assume that $\mu(\Om\setminus \Om')<\infty$,
and then $h\in L^1(\Om)$ and so $h\in \BV(\Om)$.
It follows from Proposition \ref{prop:quasisemicontinuity} that the super-level sets of $h^{\wedge}$ are $1$-quasiopen, and so using also
Lemma \ref{lem:stability of quasiopen sets}, we conclude that the set
\[
U:=(\{x\in\Om:\,h^{\wedge}(x)>1-\eps\}\cup A)\cap D
\]
is $1$-quasiopen. Clearly $A\subset U\subset D$. Defining
\[
u:=\min\{1,(1-\eps)^{-1}h\}\in\BV(\Om)
\]
we have $u^{\vee}=0$ $\mathcal H$-a.e. on $\Om'\setminus D$,
$u^{\wedge}=1$ $\mathcal H$-a.e. on $U$,
and $\Vert Du\Vert(\Om)\le (1-\eps)^{-1}(\rcapa_{\BV}(A,D,\Om')+\eps)$.

According to Lemma \ref{lem:variation measure and capacity}, there exists $\delta\in (0,\eps)$
such that whenever $H\subset \Om$ with $\capa_1(H)<\delta$,
then $\Vert Du\Vert(H)<\eps$.
Since $A$ is $1$-quasiclosed with respect to $\Om'$,
and $\Om'$, $D$, and $U$ are $1$-quasiopen, we find an open set
$G\subset \Om$ such that $\capa_1(G)<\delta/C_1^2$, $\Om'\cup G$ is open,
$D\cup G$ is open, $U\cup G$ is open,
and $A\setminus G$ is relatively closed (in the subspace
topology of $\Om'$, and then clearly
also in that of $\Om'\cup G$). By Lemma
\ref{lem:capacity and Newtonian function and bigger set} we then find a set
$V\supset G$ and a function $\eta\in N_0^{1,1}(V)$ such that $\capa_1(V)<\delta/C_1$,
$0\le\eta\le 1$ on $X$, $\eta=1$ on $G$, and $\Vert \eta\Vert_{N^{1,1}(X)}<\delta/C_1$.

Note that $u\in  \BV_0(D\cup G,\Om'\cup G)$.
By Proposition \ref{prop:weak density of lipschitz functions} we find functions
$(v_i)\subset \liploc(\Om'\cup G)$ such that $0\le v_i\le 1$,
$\supp_{\Om'\cup G}v_i\subset D\cup G$,
$v_i\to u$ in $L^1(\Om'\cup G)$, and
\[
\lim_{i\to\infty}\int_{\Om'\cup G}g_{v_i}\,d\mu=\Vert Du\Vert(\Om'\cup G).
\]
Since $A\setminus G$ is a relatively closed (in the
subspace topology of $\Om'\cup G$) subset of the open set $U\cup G$,
by Lemma \ref{lem:cutoff function lemma} we also find
a function $\rho\in \liploc(\Om'\cup G)$ such that $0\le \rho\le 1$,
$\rho=1$ on $A\setminus G$, and $\supp_{\Om'\cup G} \rho\subset U\cup G$.
Take open sets $\Om_1\Subset \Om_2\Subset \ldots \Om'\cup G$ such that $\Om'\cup G=\bigcup_{j=1}^{\infty}\Om_j$, and $\Om_0:=\emptyset$.
Denote by $L_j$ the Lipschitz constant of $\rho$ in $\Om_j$.
By Lemma \ref{lem:choosing approximating function},
from the functions $v_i$ we can construct a new function
$v\in\liploc(\Om'\cup G)$ such that $0\le v\le 1$,
$\supp_{\Om'\cup G}v\subset D\cup G$ (this follows from
\eqref{eq:definition by means of etas}),
\begin{equation}\label{eq:L1 closeness of v and u on annuli}
\Vert v-u\Vert_{L^1(\Om_{j}\setminus\Om_{j-1})}<2^{-j}\eps L_j^{-1}\quad
\textrm{for all }j\in\N,
\end{equation}
and
\begin{equation}\label{eq:estimate of gv}
\int_{\Om'\cup G} g_{v}\,d\mu\le \Vert Du\Vert(\Om'\cup G)+\eps.
\end{equation}
Then define
\[
w:= \rho(1-\eta)+(1-\rho)(1-\eta)v.
\]
Note that $0\le w\le 1$, $w=1$ on $A\setminus V$, and  $\supp_{\Om'\cup G}w\subset D$,
so that $w$ is admissible for $\rcapa_{1}(A\setminus V,D,\Om')$.
By the Leibniz rule \cite[Theorem 2.15, Lemma 2.18]{BB}, we have in $\Om'\cup G$
\begin{align*}
g_{w}
&\le \rho g_{\eta}+(1-\rho)g_{(1-\eta)v}
+g_{\rho}(1-\eta)(1-v)\\
&\le \rho g_{\eta}+(1-\rho)(g_{\eta}v
+g_{v}(1-\eta))+g_{\rho}(1-\eta)(1-v)\\
&\le g_{\eta}+g_{\eta}
+g_{v}+g_{\rho}(1-\eta)(1-v),
\end{align*}
since $\rho$ and $v$ take values between $0$ and $1$.
Since $u=1$ a.e. on $U$ and $g_{\rho}=0$ outside $U\cup G$
(see e.g. \cite[Corollary 2.21]{BB}), we have
\begin{align*}
\int_{\Om'\cup G} g_{\rho}(1-\eta)(1-v)\,d\mu
&\le \int_{\Om'\setminus G} g_{\rho}(1-v)\,d\mu\\
&=\int_{\Om'\setminus G} g_{\rho}(u-v)\,d\mu\\
&\le \sum_{j=1}^{\infty}\int_{\Om_{j}\setminus \Om_{j-1}} g_{\rho}|u-v|\,d\mu\\
&\le \sum_{j=1}^{\infty}L_j\Vert u-v\Vert_{L^1(\Om_{j}\setminus\Om_{j-1})}\\
&<\eps\quad\textrm{by }\eqref{eq:L1 closeness of v and u on annuli}.
\end{align*}
Thus
\[
\int_{\Om'} g_{w}\,d\mu
\le 2\int_{\Om'} g_{\eta}\,d\mu
+ \int_{\Om'} g_{v}\,d\mu+\eps
\le 2 \delta/C_1+\Vert Du\Vert(\Om'\cup G)+2\eps
\]
by \eqref{eq:estimate of gv}.
Thus we have
\begin{equation}\label{eq:capacity of A minus V}
\begin{split}
\rcapa_{1}(A\setminus V,D,\Om')
\le \int_{\Om'} g_{w}\,d\mu
&\le \Vert Du\Vert(\Om'\cup G)+2\delta/C_1+2\eps\\
&\le \Vert Du\Vert(\Om)+4\eps\\
&\le (1-\eps)^{-1}(\rcapa_{\BV}(A,D,\Om')+\eps)+4\eps.
\end{split}
\end{equation}
Moreover, by using Lemma \ref{lem:capacity and Newtonian function and bigger set}
again, we find a set $W\supset V$ and a function
$\xi\in N_0^{1,1}(W)$ such that $\capa_1(W)< \delta$, $0\le \xi\le 1$ on
$X$, $\xi=1$ on $V$, and $\Vert \xi\Vert_{N^{1,1}(X)}< \delta$.
Since $\xi$ is $1$-finely continuous $1$-q.e. by \cite[Corollary 5.4]{L-FC},
we have $\xi=1$ $1$-q.e. on $\overline{V}^1$ (the $1$-fine closure of $V$).
By Proposition \ref{prop:Lebesgue points for Sobolev functions open set}
and \eqref{eq:null sets of Hausdorff measure and capacity}
we have $\xi^{\wedge}=1$ $\mathcal H$-a.e. on $\overline{V}^1$, and now
clearly $(\xi u)^{\wedge}=1$ $\mathcal H$-a.e. on
$A\cap \overline{V}^1$.
Clearly also $(\xi u)^{\wedge}=(\xi u)^{\vee}=0$ $\mathcal H$-a.e. on $\Om'\setminus D$,
and $\xi u\in L^1(\Om)$. Thus $\xi u$ is admissible for
$\rcapa_{\BV}(A\cap \overline{V}^1,D,\Om')$.
By the Leibniz rule, see \cite[Proposition 4.2]{KKST3}, we get
for constant $C=C(C_d,C_P,\lambda)$
\begin{align*}
&\rcapa_{\BV}(A\cap \overline{V}^1,D,\Om')\\
&\qquad\qquad\le \Vert D(\xi u)\Vert(\Om)\\
&\qquad\qquad\le C\left(\int_\Om \xi^{\vee}\,d\Vert Du\Vert+\int_\Om u^{\vee} \,d\Vert D\xi\Vert\right)\\
&\qquad\qquad\le C\left(\int_\Om \xi\,d\Vert Du\Vert+\int_\Om u g_{\xi}\,d\mu\right)
\quad \textrm{by Proposition }\ref{prop:Lebesgue points for Sobolev functions open set},\ \eqref{eq:absolute continuity of var measure wrt H},\textrm{  and }\eqref{eq:Sobolev subclass BV}\\
&\qquad\qquad\le C\left(\Vert D u\Vert(\Om\cap W)+\int_\Om g_{\xi}\,d\mu\right)\\
&\qquad\qquad< C(\eps+ \delta);
\end{align*}
recall that $\Vert Du\Vert(W\cap \Om)<\eps$ since $\capa_1(W)<\delta$.
Combining the above with \eqref{eq:capacity of A minus V}, we conclude that for
any $\eps>0$ there is an open set $V\subset X$ such that $\capa_1(V)<\eps$,
\[
\rcapa_1(A\setminus V,D,\Om')\le \rcapa_{\BV}(A,D,\Om')+\eps,\quad\textrm{and}\quad
\rcapa_{\BV}(A\cap \overline{V}^1,D,\Om')<\eps.
\]

Fix a new $\eps>0$.
Note that $\overline{V}^1$ is $1$-finely closed and thus $1$-quasiclosed
by Theorem \ref{thm:finely open is quasiopen and vice versa}, and thus
$A\cap \overline{V}^1$ is $1$-quasiclosed with respect to $\Om'$.
Thus we can repeat the above procedure with $A$
replaced by $A\cap \overline{V}^1$. Denote $V=V_1$.
Inductively, for each $i\in\N$ we find a set $V_i$ such that
$\capa_1(V_i)< 1/i$,
\begin{equation}\label{eq:estimate for 1capacity with BVcapacity}
\rcapa_{1}\left(A\cap \bigcap_{j=1}^{i-1}\overline{V_j}^1
\setminus V_i,D,\Om'\right)
\le \rcapa_{\BV}\left(A\cap \bigcap_{j=1}^{i-1}\overline{V_j}^1,D,\Om'\right)+2^{-i}\eps,
\end{equation}
and
\begin{equation}\label{eq:estimate for small BVcapacity}
\rcapa_{\BV}\left(A\cap \bigcap_{j=1}^{i}\overline{V_j}^1,D,\Om'\right)\le 2^{-i}\eps.
\end{equation}
For each $k\in\N$, clearly
\[
A\setminus \bigcup_{i=1}^{k}
\left(A\cap \bigcap_{j=1}^{i-1}\overline{V_j}^1\setminus V_i\right)
\subset V_k,
\]
with $\capa_1(V_k)<1/k$, and so
\[
A\setminus \bigcup_{i=1}^{\infty}
\left(A\cap \bigcap_{j=1}^{i-1}\overline{V_j}^1\setminus V_i\right)
\]
is a set of $1$-capacity zero. Thus by the subadditivity of $\rcapa_1$
(see \cite[Theorem 3.4]{BB-cap})
\begin{align*}
\rcapa_1(A,D,\Om')
&\le\sum_{i=1}^{\infty}\rcapa_{1}\left(A\cap \bigcap_{j=1}^{i-1}
\overline{V_j}^1 \setminus V_i,D,\Om'\right)\\
&\le \sum_{i=1}^{\infty}\left(\rcapa_{\BV}\left(A\cap \bigcap_{j=1}^{i-1}\overline{V_j}^1,D,\Om'\right)+2^{-i}
\eps\right)\quad\textrm{by }\eqref{eq:estimate for 1capacity with BVcapacity}\\
&\le \rcapa_{\BV}(A,D,\Om')+2^{-1}\eps+\sum_{i=2}^{\infty}\left(2^{-i+1}\eps+
2^{-i}\eps\right)\quad\textrm{by }\eqref{eq:estimate for small BVcapacity}\\
&= \rcapa_{\BV}(A,D,\Om')+2\eps.
\end{align*}
Letting $\eps\to 0$, we get the result.
\end{proof}

\section{The approximation result}\label{sec:approximation}

In this section we prove our approximation result,
Theorem \ref{thm:main}. Recall that  this theorem states that
we can approximate a given $\BV$ function by $\SBV$ functions in the strict sense,
pointwise uniformly, and without adding significant jumps.
First we note that if we were to drop one of the last two conditions,
the proof would be straightforward.
Again we will always denote by $\Om$ a nonempty open set.

\begin{example}\label{ex:simpler approximations}
Let $u\in\BV(\Om)$.
From Lemma \ref{lem:choosing approximating function}
(essentially, from the definition of the total variation)
we obtain a sequence
$(u_i)\subset \SBV(\Om)$ (in fact, $(u_i)\subset \liploc(\Om)$)
such that $u_i\to u$ strictly and
$\mathcal H(S_{u_i}\setminus S_u)=0$, because in fact
$S_{u_i}=\emptyset$.
Usually, however, the $u_i$'s do not converge to $u$
uniformly, and this is in fact impossible for example when $u$ is a function
on the real line with a nonempty jump set.
Nonetheless, when $u$ is the Cantor ternary function on the unit interval
and the $u_i$'s are the usual Lipschitz functions used in its construction
(see e.g. \cite[Example 1.67]{AFP}), then also $u_i\to u$ uniformly.

On the other hand, assuming for simplicity that $\Om$ is bounded and $u$ is nonnegative,
if we define approximations
		\[
		u_i:=\frac{1}{i}\sum_{j=1}^{\infty} \ch_{\{u>t_{i,j}\}}\quad\textrm{with}\quad
		\Vert Du_i\Vert(\Om)\le \frac{1}{i}\sum_{j=1}^{\infty} P({\{u>t_{i,j}\}},\Om),
		\]
		then by the coarea formula \eqref{eq:coarea} we can see that
		with a suitable choice of the numbers $t_{i,j}\ge 0$,
		we get $u_i\to u$ strictly, and uniformly.
		However, now the jump sets $S_{u_i}$ are usually very large.
\end{example}

To prove the approximation result, we first consider a case where the function only
 has small jumps.

\begin{proposition}\label{prop:uniform approximation}
	Let $\Om'\subset \Om$ such that $\Om'$ is $1$-quasiopen
	and $\Om$ is open, and let $u\in\BV(\Om)$ and $\beta>0$ such that
	$u^{\vee}-u^{\wedge}<\beta$ in $\Om'$. Then for every $\eps>0$ there exists
	$v\in N^{1,1}(\Om')$ such that
	$\Vert v-u\Vert_{L^{\infty}(\Om')}\le 4\beta$ and
	\[
	\int_{\Om'}g_v\,d\mu\le \Vert Du\Vert(\Om')+\eps.
	\]
\end{proposition}

\begin{proof}
First assume that $u\ge 0$.
Fix $\eps>0$.
For each $i\in\N$, let
\[
A_i:=\{x\in\Om':\,u^{\vee}(x)\ge (i+1)\beta\}
\quad\textrm{and}\quad
D_i:=\{x\in\Om':\,u^{\wedge}(x)> (i-2)\beta\}.
\]
By Proposition \ref{prop:quasisemicontinuity}, each $A_i$ is $1$-quasiclosed with respect to $\Om'$.
It is straightforward to check that the intersection of two
$1$-quasiopen sets is $1$-quasiopen, and so
each $D_i$ is $1$-quasiopen (with respect to $X$).
Moreover, for all $i\in\N$,
\[
A_i
\subset \{x\in\Om':\,u^{\wedge}(x)\ge i\beta\}\\
\subset \{x\in\Om':\,u^{\vee}(x)> (i-1)\beta\}\subset D_i.
\]
Fix $i\in\N$. Let
\[
u_i:=\frac{1}{\beta}\min\{\beta,(u-(i-1)\beta)_+\}\in\BV(\Om).
\]
Then
\begin{align*}
\rcapa_{\BV}(A_i,D_i,\Om')
&\le \rcapa_{\BV}(\{u^{\wedge}\ge i\beta\}\cap\Om',
\{u^{\vee}> (i-1)\beta\}\cap\Om',\Om')\\
&\le \Vert Du_i\Vert(\Om')\\
&=\frac{1}{\beta}\int_{(i-1)\beta}^{i\beta}P(\{u>t\},\Om')\,dt
\end{align*}
by the coarea formula \eqref{eq:coarea}, which also applies to $1$-quasiopen sets,
see \cite[Proposition 3.8]{L-LSC}.
By Theorem \ref{thm:1capacity and BVcapacity are equal}
we find a function $v_i\in N_0^{1,1}(D_i,\Om')$ such that
$v_i=1$ on $A_i$ and
\[
\int_{\Om'}g_{v_i}\,d\mu<\rcapa_{\BV}(A_i,D_i,\Om')+\frac{2^{-i}\eps}{\beta}
\le\frac{1}{\beta}\int_{(i-1)\beta}^{i\beta}P(\{u>t\},\Om')\,dt
+\frac{2^{-i}\eps}{\beta}.
\]
Now define
\[
v:=\beta\sum_{i=3}^{\infty}v_i.
\]
It is easy to check that
$u^{\vee}-4\beta\le v\le u^{\wedge}$ on $\Om'$, so that
$\Vert v-u\Vert_{L^{\infty}(\Om')}\le 4\beta$
and also $v\in L^1(\Om')$.
Since $g_v\le \beta\sum_{i=3}^{\infty} g_{v_i}$ (see e.g. \cite[Lemma 1.52]{BB}),
\begin{align*}
\int_{\Om'} g_v\,d\mu
&\le \beta\sum_{i=1}^{\infty}\int_{\Om'} g_{v_i}\,d\mu\\
&\le\beta\sum_{i=1}^{\infty}\left(\frac{1}{\beta}
\int_{(i-1)\beta}^{i\beta}P(\{u>t\},\Om')\,dt+\frac{2^{-i}\eps}{\beta}\right)\\
&=\int_{0}^{\infty}P(\{u>t\},\Om')\,dt+\eps\\
&=\Vert Du\Vert(\Om')+\eps.
\end{align*}
This completes the proof in the case $u\ge 0$.

In the general case, we find a function $w_1\in N^{1,1}(\Om')$ corresponding to $u_+$ and
a function $w_2\in N^{1,1}(\Om')$ corresponding to $u_-$. Then for
$v:=w_1-w_2\in N^{1,1}(\Om')$ we have
$\Vert v-u\Vert_{L^{\infty}(\Om')}\le 4\beta$ and
\begin{align*}
\int_{\Om'}g_{v}\,d\mu\le \int_{\Om'}g_{w_1}\,d\mu+\int_{\Om'}g_{w_2}\,d\mu
&\le \Vert Du_+\Vert(\Om')+\Vert Du_-\Vert(\Om')+2\eps\\
&=\Vert Du\Vert(\Om')+2\eps,
\end{align*}
where the last inequality follows from the coarea formula.
\end{proof}

Now we consider the more general case where
$u$ may also have large jumps.

\begin{proposition}\label{prop:uniform approximation for bounded functions}
Let $u\in \BV(\Om)$ and let $\eps>0$.
Then there exists
$v\in \BV(\Om)$ such that $\Vert v-u\Vert_{L^{\infty}(\Om)}\le \eps$,
$\Vert Dv\Vert(\Om)\le \Vert Du\Vert(\Om)+\eps$,
$\Vert Dv\Vert^c(\Om)=0$,
and $\mathcal H(S_v\setminus S_u)=0$.
\end{proposition}

\begin{proof}
Fix $0<\delta<\min\{1,\eps\}/4$ to be chosen later.
Let $S:=\{x\in \Om:\, u^{\vee}-u^{\wedge}\ge \delta\}$.
By Proposition \ref{prop:quasisemicontinuity}, $\Om\setminus S$ is a $1$-quasiopen set.
Apply Lemma \ref{prop:uniform approximation} to find a function
$v\in N^{1,1}(\Om\setminus S)$ such that
$\Vert v-u\Vert_{L^{\infty}(\Om\setminus S)}\le 4\delta$
and
\begin{equation}\label{eq:estimate for gv in Om minus S}
\int_{\Om\setminus S}g_v\,d\mu\le \Vert Du\Vert(\Om\setminus S)+\eps.
\end{equation}
By the decomposition \eqref{eq:variation measure decomposition} it is clear that
$\mathcal H(S)<\infty$, from which it easily follows that $\mu(S)=0$.
Thus we have in fact $\Vert v-u\Vert_{L^{\infty}(\Om)}\le \eps$
and $v\in L^1(\Om)$, as desired.

Now we estimate $\Vert Dv\Vert(\Om)$.
Take a sequence $(u_i)\subset N^{1,1}(\Om)$
(from Lemma \ref{lem:choosing approximating function})
such that $u_i\to u$ in $L^1(\Om)$ and
$\int_{\Om}g_{u_i}\,d\mu\to\Vert Du\Vert(\Om)$.
Then $v-u_i\to v-u$ in $L^1(\Om)$.
Letting $w_i:=\min\{1,\max\{-1,v-u_i\}\}$, we have
$w_i\to v-u$ in $L^1(\Om)$.
Let $i\in\N$ be fixed.
We find a covering $\{B_j=B(x_j,r_j)\}_{j=1}^{\infty}$
such that $r_j\le 1/i$ for all $j$, $S\subset \bigcup_{j=1}^{\infty}B_j$, and
\begin{equation}\label{eq:covering of boundary of Omega}
\sum_{j=1}^{\infty} \frac{\mu(B(x_j,r_j))}{r_j}<\mathcal H(S)+1/i.
\end{equation}
Then pick $1/r_j$-Lipschitz functions $\eta_j$ such that
$0\le \eta_j\le 1$ on $X$, $\eta_j=1$ on $B(x_j,r_j)$, and $\eta_j=0$ outside $B(x_j,2r_j)$.
Define
$\rho_i:= \sup_{j\in \N} \eta_j$.
Consider the function
\[
h_i:=(1-\rho_i)w_i.
\]
Let $g\in L^1(\Om\setminus S)$ be a $1$-weak upper gradient of $w_i$ in
$\Om\setminus S$; for example $g_{v}+g_{u_i}$ will do.
By \cite[Corollary 2.21]{BB} we know that 
$\ch_{2B_j}/r_j$ is a $1$-weak upper gradient of $\eta_j$ (in $X$), and then
$\sum_{j=1}^{\infty}\frac{\ch_{2B_j}}{r_j}$ is a $1$-weak upper gradient of $\rho_i$
(in $X$) by e.g. \cite[Lemma 1.52]{BB}.
We show that
\[
g_i:= g+\sum_{j=1}^{\infty}\frac{\ch_{2B_j}}{r_j}
\]
is a $1$-weak upper gradient of $h_i$ in $\Om$.
By the Leibniz rule \cite[Theorem 2.15]{BB},
$g_i$ is a $1$-weak upper gradient of
$h_i$ in $\Omega\setminus S$; recall that $\Vert w_i\Vert_{L^{\infty}(\Om)}\le 1$.
Take a curve $\gamma$ in $\Om$ such that the upper gradient inequality is satisfied by
$h_i$ and $g_i$ on all subcurves of $\gamma$ in $\Om\setminus S$;
this is true for $1$-a.e. curve in $\Om$ by \cite[Lemma 1.34]{BB}.
Moreover, the fact that $\Om\setminus S$ is $1$-quasiopen implies 
by \cite[Remark 3.5]{S2} that it is also \emph{$1$-path open},
meaning that for $1$-a.e. curve $\gamma$, the set
$\gamma^{-1}(\Om\setminus S)$ is a relatively open subset of $[0,\ell_{\gamma}]$.
Thus we can assume that $\gamma^{-1}(S)$ is a compact subset
of the relatively open set 
$\gamma^{-1}\big(\bigcup_{j=1}^{\infty} B_j\big)$.
Thus
$\gamma$ can be split into a finite number of subcurves each of which lies either entirely in $\bigcup_{j=1}^{\infty} B_j$, or entirely in
$\Omega\setminus S$.
If $\gamma_1$ is a subcurve lying entirely in $\bigcup_{j=1}^{\infty} B_j$,
\[
|h_i(\gamma_1(0))-h_i(\gamma_1(\ell_{\gamma_1}))|=|0-0|=0,
\]
so the upper gradient inequality is satisfied. If $\gamma_2$ is a subcurve lying entirely in $\Omega\setminus S$, then
\[
|h_i(\gamma_2(0))-h_i(\gamma_2(\ell_{\gamma_2}))|\le \int_{\gamma_2}g_i\,ds
\]
by our choice of $\gamma$. Summing over the subcurves, we obtain
\[
|h_i(\gamma(0))-h_i(\gamma(\ell_{\gamma}))|\le \int_{\gamma}g_i\,ds.
\]
Thus $g_i$ is a $1$-weak upper gradient of $h_i$ in $\Om$.
By \eqref{eq:covering of boundary of Omega} we have
\begin{equation}\label{eq:L1 estimate for gi}
\Vert g_i\Vert_{L^1(\Om)}\le \Vert g\Vert_{L^1(\Om\setminus S)}
+C_d(\mathcal H(S)+1/i).
\end{equation}
Since $-1\le w_i\le 1$,
\[
\Vert h_i-w_i\Vert_{L^1(\Om)}=\Vert \rho_i w_i\Vert_{L^1(\Om)}
\le \Vert \rho_i \Vert_{L^1(\Om)}\le \frac{1}{i}\sum_{j=1}^{\infty}
\frac{\mu(2B_j)}{r_j}\le \frac{C_d}{i}(\mathcal H(S)+1/i).
\]
Recall that $w_i\to v-u$ in $L^1(\Om)$.
Thus also $h_i\to v-u$ in $L^1(\Om)$, and so
by \eqref{eq:L1 estimate for gi}
\[
\Vert D(v-u)\Vert(\Om)\le \liminf_{i\to\infty}\Vert g_i\Vert_{L^1(\Om)}\le 
\Vert g\Vert_{L^1(\Om\setminus S)}+C_d \mathcal H(S)<\infty.
\]
Thus also $\Vert Dv\Vert(\Om)<\infty$ (recall
\eqref{eq:BV functions form vector space}).
By the decomposition \eqref{eq:variation measure decomposition}
and the discussion after it,
we find that only the jump part of $\Vert D(v-u)\Vert$ can charge $S$,
and then from the fact that
$\Vert v-u\Vert_{L^{\infty}(\Om)}\le 4\delta$ we get
\begin{equation}\label{eq:Dv minus u in S}
\begin{split}
\Vert D(v-u)\Vert(S)
\le C_d\int_S ((v-u)^{\vee}-(v-u)^{\wedge})\,d\mathcal H
\le 8C_d\delta\mathcal H(S).
\end{split}
\end{equation}
By another application of the decomposition \eqref{eq:variation measure decomposition},
\[
\infty>\Vert Du\Vert(S_u)\ge \alpha \int_{S_u} (u^{\vee}-u^{\wedge})\,d\mathcal H
=\alpha\int_0^{\infty}\mathcal H(\{u^{\vee}-u^{\wedge}>t\})\,dt
\]
by Cavalieri's principle.
Since the function $t\mapsto \mathcal H(\{u^{\vee}-u^{\wedge}>t\})$
is thus integrable, necessarily
\[
\liminf_{t\to 0}t\mathcal H(\{u^{\vee}-u^{\wedge}>t\})=0.
\]
Thus by choosing a suitable small $\delta$, we can ensure that
$\delta\mathcal H(S)<\eps/(8 C_d)$. Hence \eqref{eq:Dv minus u in S} gives
$\Vert D(v-u)\Vert(S)\le \eps$ and so
\[
\Vert Dv\Vert(S)\le \Vert Du\Vert(S)+\eps.
\]
Thus we get (note that $1$-quasiopen sets can be
seen to be $\Vert Du\Vert$-measurable
by Lemma \ref{lem:variation measure and capacity})
\begin{align*}
\Vert Dv\Vert(\Om)
&=\Vert Dv\Vert(\Om\setminus S)+\Vert Dv\Vert(S)\\
&\le \int_{\Om\setminus S}g_v\,d\mu+\Vert Dv\Vert(S)\quad\textrm{by Theorem }
\ref{thm:characterization of total variational}\\
&\le \Vert Du\Vert(\Om\setminus S)+\eps+\Vert Du\Vert(S)+\eps
\quad\textrm{by }\eqref{eq:estimate for gv in Om minus S}\\
&= \Vert Du\Vert(\Om)+2\eps,
\end{align*}
as desired.
Note that for any $A\subset \Om\setminus S$ with $\mu(A)=0$,
by Theorem \ref{thm:characterization of total variational} we have for any open $W$
with $A\subset W\subset \Om$ that
\[
\Vert Dv\Vert(A)\le \Vert Dv\Vert(W\setminus S)\le\int_{W\setminus S}g_v\,d\mu,
\]
which becomes arbitrarily small by choosing $\mu(W)$ small.
We conclude that $\Vert Dv\Vert^c(\Om\setminus S)=0$, and thus
$\Vert Dv\Vert^c(\Om)=0$ since $\mathcal H(S)<\infty$.

By Theorem \ref{thm:quasicontinuity}, $v$ is $1$-quasicontinuous on
$\Om\setminus S$, so by \cite[Theorem 5.1]{L-NC} it is also
$1$-finely continuous $1$-q.e. on $\Om\setminus S$, and so
by \eqref{eq:thinness and measure thinness} clearly
$v^{\wedge}=v^{\vee}$ $1$-q.e. on $\Om\setminus S$.
Hence
$\mathcal H(S_v\setminus S)=0$ and so 
$\mathcal H(S_v\setminus S_{u})=0$.
\end{proof}

To obtain the strongest possible result, we will apply
the above proposition only in a small open subset of $\Om$ where
the Cantor part of $\Vert Du\Vert$ is concentrated. For this,
we will need the following
extension lemma.

\begin{lemma}\label{lem:extension BV function in quasiopen set}
	Let $W\subset \Om\subset X$ be open sets,
	let $u\in \BV(W)$, and suppose that
	\begin{equation}\label{eq:extension limit condition}
	\lim_{W\ni y\to x}|u|^{\vee}(y)= 0
	\end{equation}
	for all $x\in \partial W$.
	Then $u\in\BV_0(W,\Om)$.
\end{lemma}

\begin{proof}
First assume that $\supp_X u\subset W$.
Then $u\in\BV_0(W,\Om)$ with $\Vert Du\Vert(W)= \Vert Du\Vert(\Om)$
by Lemma \ref{lem:extension of compactly supported function}.

In the general case, note that for the functions
\[
u_{\delta}:=(u-\delta)_+-(u+\delta)_-,\quad \delta>0,
\]
we have $\supp_X u_{\delta}\subset W$.
Thus, understanding $u$ to be zero extended to $\Om\setminus W$,
we have $u_{\delta}\to u$ in $L^1(\Om)$ and then
\[
\Vert Du\Vert(\Om)\le \liminf_{i\to\infty}\Vert Du_{\delta}\Vert(\Om)
= \liminf_{i\to\infty}\Vert Du_{\delta}\Vert(W)\le \Vert Du\Vert(W),
\]
so that $u\in\BV(\Om)$. By \eqref{eq:extension limit condition},
clearly $u^{\wedge}=u^{\vee}=0$ on $\Om\setminus W$, and so $u\in\BV_0(W,\Om)$.
\end{proof}

Now we prove our main approximation result, which we first give in the following
form.

\begin{theorem}\label{thm:approximation}
Let $u\in\BV(\Om)$ and let $\eps>0$.
Then there exists $w\in\SBV(\Om)$ and an open set
$W\subset \Om$ such that $w\ge u$ on $\Om$, $w=u$ on $\Om\setminus W$,
$\Vert w-u\Vert_{L^1(\Om)}<\eps$,
$\Vert w-u\Vert_{L^{\infty}(\Om)}<\eps$,
$\mathcal H(S_w\setminus S_u)=0$,
$\mu(W)<\eps$, $\Vert Du\Vert(W)< \Vert Du\Vert^c(\Om)+\eps$,
$\Vert D(w-u)\Vert(\Om\setminus W)=0$,
$\Vert Dw\Vert(W)<\Vert Du\Vert(W)+\eps$,
and
\begin{equation}\label{eq:w-u limit}
\lim_{W\ni y\to x}|w-u|^{\vee}(y)=0\quad\textrm{for all }x\in\partial W.
\end{equation}
\end{theorem}
\begin{proof}
By the decomposition \eqref{eq:variation measure decomposition},
we find a Borel set $A\subset\Om$ such that $\mu(A)=0$ and
$\Vert Du\Vert^c(A)=\Vert Du\Vert(A)=\Vert Du\Vert^c(\Om)$.
Take an open set $W\subset\Om$ such that $W\supset A$ and
$\Vert Du\Vert(W)<\Vert Du\Vert^c(\Om)+\eps$, and $\mu(W)<\eps$.
By Proposition \ref{prop:uniform approximation for bounded functions}
we find a sequence $(u_i)\subset \BV(W)$ such that
$\Vert u_i-u\Vert_{L^{\infty}(W)}\to 0$,
\[
\lim_{i\to\infty}\Vert Du_i\Vert(W)= \Vert Du\Vert(W),
\]
$\Vert D u_i\Vert^c(W)=0$, and $\mathcal H(W\cap S_{u_i}\setminus S_u)=0$.
Then by Lemma \ref{lem:BV pasting lemma}
we find a function $v\in \BV(W)$ such that $v\ge u$ on $W$, $\Vert v-u\Vert_{L^1(W)}<\eps$,
	$\Vert v-u\Vert_{L^{\infty}(W)}<\eps$,
	\[
	\Vert Dv\Vert(W)<\Vert Du\Vert(W)+\eps,
	\]
	\begin{equation}\label{eq:v-u limit}
	\lim_{W\ni y\to x}|v-u|^{\vee}(y)=0\quad\textrm{for all }x\in \partial W,
	\end{equation}
	$\Vert D v\Vert^c(W)=0$, and $\mathcal H(W\cap S_v\setminus S_u)=0$.
Let
\[
w:=
\begin{cases}
u &\textrm{ on }\Om\setminus W,\\
v &\textrm{ on }W.
\end{cases}
\]
Clearly $w\ge u$, $\Vert w-u\Vert_{L^1(\Om)}<\eps$, and
	$\Vert w-u\Vert_{L^{\infty}(\Om)}<\eps$.
By Lemma \ref{lem:extension BV function in quasiopen set} and \eqref{eq:v-u limit},
$w-u\in\BV(\Om)$ and then $w\in\BV(\Om)$.
Equation \eqref{eq:v-u limit} gives \eqref{eq:w-u limit}.
By \eqref{eq:w-u limit}, $\partial^*\{w-u>t\}\setminus W=\emptyset$
for all $t\neq 0$.
Thus by the coarea formula \eqref{eq:coarea} and \eqref{eq:def of theta},
\begin{equation}\label{eq:w and u close in BV energy norm}
\begin{split}
\Vert D(w-u)\Vert(\Om\setminus W)
&=\int_{-\infty}^{\infty}P(\{w-u>t\},\Om\setminus W)\,dt\\
&\le C_d\int_{-\infty}^{\infty}\mathcal H(\partial^*\{w-u>t\}\setminus W)\,dt\\
&=0.
\end{split}
\end{equation}
Also
\[
\Vert Dw\Vert(W)=\Vert Dv\Vert(W)<\Vert Du\Vert(W)+\eps
\]
and
\begin{align*}
\Vert Dw\Vert^c(\Om)
&=\Vert Dw\Vert^c(W)+\Vert Dw\Vert^c(\Om\setminus W)\\
&=\Vert Dv\Vert^c(W)+\Vert Du\Vert^c(\Om\setminus W)\quad\textrm{by }
\eqref{eq:w and u close in BV energy norm}\\
&=0.
\end{align*}
Equation \eqref{eq:w-u limit} also implies
that $w^{\wedge}=u^{\wedge}$ and $w^{\vee}=u^{\vee}$ on $\Om\setminus W$, so that
$S_w\setminus W=S_u\setminus W$. We have
$\mathcal H(W\cap S_v\setminus S_u)=0$ and so
also $\mathcal H(W\cap S_w\setminus S_u)=0$.
We conclude that $\mathcal H(S_w\setminus S_u)=0$.
\end{proof}

Next we show the sharpness of the condition
$\Vert D(w-u)\Vert(\Om)<2\Vert Du\Vert^c(\Om)+\eps$;
in particular this demonstrates the fact that
it is generally impossible to approximate
$\BV$ functions by $\SBV$ functions in the $\BV$ \emph{norm}.

\begin{example}\label{ex:sharpness of BV norm closeness}
	Let $u\in\BV(\Om)$ and let
	$(u_i)\subset\SBV(\Om)$ such that $u_i\to u$ in $L^1(\Om)$.
	We show that necessarily
	\[
	\liminf_{i\to\infty}\Vert D(u_i-u)\Vert(\Om)\ge 2\Vert Du\Vert^c(\Om).
	\]
	Fix $\eps>0$.
	We find a Borel set $F\subset \Om$ such that $\mu(F)=0$ and
	$\Vert Du\Vert^c(F)=\Vert Du\Vert(F)=\Vert Du\Vert^c(\Om)$.
	We also find an open set $W\subset \Om$ such that $W\supset F$ and
	$\Vert Du\Vert(W)< \Vert Du\Vert(F)+\eps$.
	Let $S:=\bigcup_{i=1}^{\infty}S_{u_i}$ and $H:=F\setminus S$.
	Since $S$ is $\sigma$-finite with respect to $\mathcal H$, $\Vert Du\Vert^c(S)=0$.
	Then
	\begin{align*}
	\Vert D(u-u_i)\Vert(H)
	&\ge \Vert Du\Vert(H)-\Vert Du_i\Vert(H)\\
	&=\Vert Du\Vert(F)-\Vert Du_i\Vert(H)\\
	&=\Vert Du\Vert^c(\Om)-0
	\end{align*}
	for all $i\in\N$. Moreover,
	\begin{align*}
	\Vert D(u-u_i)\Vert(W\setminus H)
	&\ge \Vert Du_i\Vert(W\setminus H)-\Vert Du\Vert(W\setminus H)\\
	&\ge\Vert Du_i\Vert(W\setminus H)-\eps\\
	&=\Vert Du_i\Vert(W)-\eps
	\end{align*}
	for all $i\in\N$, and thus
	\begin{align*}
	\liminf_{i\to\infty}\Vert D(u-u_i)\Vert(W\setminus H)
	&\ge \liminf_{i\to\infty}\Vert Du_i\Vert(W)-\eps\\
	&\ge \Vert Du\Vert(W)-\eps\quad \textrm{since }u_i\to u\textrm{ in }L^1(\Om)\\
	&\ge \Vert Du\Vert^c(\Om)-\eps.
	\end{align*}
	In total, we get
	\begin{align*}
	\liminf_{i\to\infty}\Vert D(u-u_i)\Vert(\Om)
	&\ge\liminf_{i\to\infty}\Vert D(u-u_i)\Vert(W\setminus H)
	+\liminf_{i\to\infty}\Vert D(u-u_i)\Vert(H)\\
	&\ge 2\Vert Du\Vert^c(\Om)-\eps,
	\end{align*}
	and so we have the result.
\end{example}

Now we get the following corollary,
which in particular implies Theorem \ref{thm:main}.

\begin{corollary}\label{cor:approximation result}
Let $u\in\BV(\Om)$. Then there exists a
sequence $(u_i)\subset \SBV(\Om)$
such that
\begin{itemize}
\item $u_i\to u$ in $L^1(\Om)$ and $\Vert Du_i\Vert(\Om)\to \Vert Du\Vert(\Om)$,
\item $\lim_{i\to\infty}\Vert D(u_i-u)\Vert(\Om)= 2\Vert Du\Vert^c(\Om)$,
\item $\limsup_{i\to\infty}\Vert Du\Vert(\{|u_i-u|^{\vee}\neq 0\})
\le \Vert Du\Vert^c(\Om)$,
$\lim_{i\to\infty}\mu(\{|u_i-u|^{\vee}\neq 0\})=0$,
\item $u_i\ge u$ and $u_i\to u$ uniformly in $\Om$,
\item $\mathcal H(S_{u_i}\setminus S_u)=0$ for all $i\in\N$, and
\item $\lim_{\Om\ni y\to x}|u_i-u|^{\vee}(y)=0$
for all $x\in\partial\Om$.
\end{itemize}
\end{corollary}

\begin{proof}
This follows almost directly from Theorem \ref{thm:approximation} and
Example \ref{ex:sharpness of BV norm closeness}.
The third condition follows from \eqref{eq:w-u limit}
and the estimates $\mu(W)<\eps$ and $\Vert Du\Vert(W)< \Vert Du\Vert^c(\Om)+\eps$
given in Theorem \ref{thm:approximation}.
The last condition also follows from \eqref{eq:w-u limit}.
\end{proof}

Note that the first condition says that the $u_i$'s converge to $u$ strictly,
the second condition
describes closeness in the $\BV$ norm, and the third condition implies that
\[
\limsup_{i\to\infty}\Vert Du\Vert(\{u_i^{\wedge}\neq u^{\wedge}\}\cup
\{u_i^{\vee}\neq u^{\vee}\})\le \Vert Du\Vert^c(\Om),
\]
so it describes approximation in the Lusin sense. The last condition
expresses the fact that $u_i$ and $u$ have the same ``boundary values''.

In closing, let us consider a few new capacities defined similarly as in
Definition \ref{def:variational capacities}.

\begin{definition}
	Let $A\subset D\subset \Om\subset X$ be nonempty sets
	with $\Om$ open.
	We define the variational $\SBV$-capacity by
	\[
	\rcapa_{\SBV}(A,D,\Om):=\inf \Vert Du\Vert(\Om),
	\]
	where the infimum is taken over functions $u\in \BV_0(D,\Om)\cap \SBV(\Om)$
	such that $u^{\wedge}\ge 1$ $\mathcal H$-a.e. on $A$.
	
	We define the variational \emph{diffuse} $\BV$-capacity by
	\[
	\rcapa_{\DBV}(A,D,\Om):=\inf \Vert Du\Vert(\Om),
	\]
	where the infimum is taken over functions $u\in \BV_0(D,\Om)$
	such that $\mathcal H(S_u)=0$ and
	$u^{\wedge}\ge 1$ $\mathcal H$-a.e. on $A$.
\end{definition}
	
One can also replace $\Om$ by a more general set, but we choose to consider
the above simpler case here.	
	
	\begin{corollary}
	We have
	\[
	\rcapa_{\SBV}(A,D,\Om)=\rcapa_{\BV}(A,D,\Om)\quad\textrm{and}\quad
	\rcapa_{\DBV}(A,D,\Om)=\rcapa_{1}(A,D,\Om).
	\]
	\end{corollary}
	
	\begin{proof}
	To prove the first equality,
	we can assume that $\rcapa_{\BV}(A,D,\Om)<\infty$. Let $0<\eps<1/2$.
	Take a function $u$ that is admissible for $\rcapa_{\BV}(A,D,\Om)$
	such that $\Vert Du\Vert(\Om)<\rcapa_{\BV}(A,D,\Om)+\eps$.
	By Corollary \ref{cor:approximation result} we find a function $w\in\SBV(\Om)$
	such that $\Vert Dw\Vert(\Om)<\Vert Du\Vert(\Om)+\eps$
	and $\Vert w-u\Vert_{L^{\infty}(\Om)}<\eps$.
	Then $v:=(w-\eps)_+/(1-2\eps)$ is admissible for $\rcapa_{\SBV}(A,D,\Om)$
	and so
	\[
	\rcapa_{\SBV}(A,D,\Om)\le \Vert Dv\Vert(\Om)
	\le \frac{\Vert Du\Vert(\Om)+\eps}{1-2\eps}
	\le \frac{\rcapa_{\BV}(A,D,\Om)+2\eps}{1-2\eps}.
	\]
	Letting $\eps\to 0$, the first inequality follows.
	
	To prove the second equality,
	we can assume that $\rcapa_{\DBV}(A,D,\Om)<\infty$. Let $0<\eps<1/2$.
	Take a function $u$ that is admissible for $\rcapa_{\DBV}(A,D,\Om)$
	such that $\Vert Du\Vert(\Om)<\rcapa_{\DBV}(A,D,\Om)+\eps$.
	Apply Proposition \ref{prop:uniform approximation}
	with the choice $\Om'=\Om\setminus S_u$ to find a function
	$w\in N^{1,1}(\Om')$
	such that $\int_{\Om'}g_w\,d\mu<\Vert Du\Vert(\Om)+\eps$
	and $\Vert w-u\Vert_{L^{\infty}(\Om)}<\eps$.
	Since $\mathcal H(S_u)=0$ and thus $\capa_1(S_u)=0$,
	we have in fact
	$w\in N^{1,1}(\Om)$
	with $\int_{\Om}g_w\,d\mu<\Vert Du\Vert(\Om)+\eps$,
	see \cite[Proposition 1.48]{BB}.
	Then $v:=(w-\eps)_+/(1-2\eps)$ is admissible for $\rcapa_{1}(A,D,\Om)$
	and so
	\[
	\rcapa_{1}(A,D,\Om)\le \int_{\Om}g_v\,d\mu\le \frac{\Vert Du\Vert(\Om)+\eps}{1-2\eps}
	\le \frac{\rcapa_{\DBV}(A,D,\Om)+2\eps}{1-2\eps}.
	\]
	Letting $\eps\to 0$, the second inequality follows.
	\end{proof}
	
	Note that for the first equality we did not actually need the full strength of
	our approximation result;
	recall Example \ref{ex:simpler approximations}.
	However, with our result it is also possible to handle
	much more general energies than simply $\Vert Du\Vert(\Om)$,
	given for example by convex functionals
	of linear growth, or involving terms such as $\mathcal H(S_u)$
	(like for example in the Mumford-Shah functional).
	Generally, the implication is that the absolutely continuous and jump
	parts help to optimize energy
	--- in particular, it is possible to have $\rcapa_{\BV}(A,D)<\rcapa_{1}(A,D)$,
	see \cite[Example 4.27]{L-ZB} ---
	but the Cantor part does not.

\noindent Address:\\

\noindent University of Jyvaskyla\\
Department of Mathematics and Statistics\\
P.O. Box 35, FI-40014 University of Jyvaskyla\\
E-mail: {\tt panu.k.lahti@jyu.fi}


\begin{thebibliography}{99}

\bibitem{ADC}M. Amar and V. De Cicco,
\textit{A new approximation result for BV-functions}, 
C. R. Math. Acad. Sci. Paris 340 (2005), no. 10, 735--738. 

\bibitem{A1}L. Ambrosio,
\textit{Fine properties of sets of finite perimeter in doubling metric measure spaces},
Calculus of variations, nonsmooth analysis and related topics.
Set-Valued Anal. 10 (2002), no. 2-3, 111--128.

\bibitem{ADG}L. Ambrosio and E. De Giorgi,
\textit{New functionals in the calculus of variations},
Atti Accad. Naz. Lincei Rend. Cl. Sci. Fis. Mat. Natur. (8) 82 (1988), no. 2, 199--210 (1989).

\bibitem{AFP}L. Ambrosio, N. Fusco, and D. Pallara,
\textit{Functions of bounded variation and free discontinuity problems.}
Oxford Mathematical Monographs. The Clarendon Press, Oxford University Press, New York, 2000.

\bibitem{AMP}L. Ambrosio, M. Miranda, Jr., and D. Pallara,
\textit{Special functions of bounded variation in doubling metric measure spaces},
Calculus of variations: topics from the mathematical heritage of E. De Giorgi, 1--45,
Quad. Mat., 14, Dept. Math., Seconda Univ. Napoli, Caserta, 2004.

\bibitem{BB}A. Bj\"orn and J. Bj\"orn,
\textit{Nonlinear potential theory on metric spaces},
EMS Tracts in Mathematics, 17. European Mathematical Society (EMS), Z\"urich, 2011. xii+403 pp.

\bibitem{BB-OD}A. Bj\"orn and J. Bj\"orn,
\textit{Obstacle and Dirichlet problems on arbitrary nonopen sets in metric spaces, and fine topology},
Rev. Mat. Iberoam. 31 (2015), no. 1, 161--214.

\bibitem{BB-cap}A. Bj\"orn and J. Bj\"orn,
\textit{The variational capacity with respect to nonopen sets in metric spaces},
Potential Anal. 40 (2014), no. 1, 57--80.

\bibitem{BBM-QO}A. Bj\"orn, J. Bj\"orn, and J. Mal\'y,
\textit{Quasiopen and p-path open sets, and characterizations of quasicontinuity},
Potential Anal. 46 (2017), no. 1, 181--199. 

\bibitem{BBS}A. Bj\"orn, J. Bj\"orn, and N. Shanmugalingam,
\textit{Quasicontinuity of Newton-Sobolev functions and density of
Lipschitz functions on metric spaces}, 
Houston J. Math. 34 (2008), no. 4, 1197--1211.

\bibitem{BCP}A. Braides and V. Chiad\`o Piat,
\textit{Integral representation results for functionals defined on $\SBV(\Om;\R^m)$},
J. Math. Pures Appl. (9) 75 (1996), no. 6, 595--626. 

\bibitem{CT}G. Cortesani and R. Toader,
\textit{A density result in SBV with respect to non-isotropic energies}, 
Nonlinear Anal. 38 (1999), no. 5, Ser. B: Real World Appl., 585--604. 

\bibitem{dPFP}G. de Philippis, N Fusco, and A. Pratelli,
\textit{On the approximation of SBV functions}, 
Atti Accad. Naz. Lincei Rend. Lincei Mat. Appl. 28 (2017), no. 2, 369--413. 

\bibitem{EvaG92}L. C. Evans and R. F. Gariepy,
\textit{Measure theory and fine properties of functions},
Studies in Advanced Mathematics series, CRC Press, Boca Raton, 1992.

\bibitem{Fed}H. Federer,
\textit{Geometric measure theory},
Die Grundlehren der mathematischen Wissenschaften, Band 153 Springer-Verlag New York Inc., New York 1969 xiv+676 pp. 

\bibitem{Giu84}E. Giusti,
\textit{Minimal surfaces and functions of bounded variation},
Monographs in Mathematics, 80. Birkh\"auser Verlag, Basel, 1984. xii+240 pp.

\bibitem{Hj}P. Haj\l{}asz,
\textit{Sobolev spaces on metric-measure spaces},
Heat kernels and analysis on manifolds, graphs, and metric spaces (Paris, 2002), 173--218,
Contemp. Math., 338, Amer. Math. Soc., Providence, RI, 2003.

\bibitem{HKLS}H. Hakkarainen, R. Korte, P. Lahti, and N. Shanmugalingam,
\textit{Stability and continuity of functions of least gradient},
Anal. Geom. Metr. Spaces 3 (2015), 123--139.

\bibitem{HaKi}H. Hakkarainen and J. Kinnunen,
\textit{The BV-capacity in metric spaces},
Manuscripta Math. 132 (2010), no. 1-2, 51--73.

\bibitem{HaSh}H. Hakkarainen and N. Shanmugalingam,
\textit{Comparisons of relative BV-capacities and Sobolev capacity in metric spaces},
Nonlinear Anal. 74 (2011), no. 16, 5525--5543.

\bibitem{HKM}J. Heinonen, T. Kilpel\"ainen, and O. Martio,
\textit{Nonlinear potential theory of degenerate elliptic equations},
Unabridged republication of the 1993 original. Dover Publications, Inc., Mineola, NY, 2006. xii+404 pp.

\bibitem{HK}J. Heinonen and P. Koskela,
\textit{Quasiconformal maps in metric spaces with controlled geometry},
Acta Math. 181 (1998), no. 1, 1--61.

\bibitem{KKST2}J. Kinnunen, R. Korte, N. Shanmugalingam, and H. Tuominen,
\textit{Lebesgue points and capacities via the boxing inequality in metric spaces},
Indiana Univ. Math. J. 57 (2008), no. 1, 401--430. 

\bibitem{KKST3}J. Kinnunen, R. Korte, N. Shanmugalingam, and H. Tuominen,
\textit{Pointwise properties of functions of bounded variation in metric spaces},
Rev. Mat. Complut. 27 (2014), no. 1, 41--67.

\bibitem{KKST-DG}J. Kinnunen, R. Korte, N. Shanmugalingam, and H. Tuominen,
\textit{The De Giorgi measure and an obstacle problem related to minimal surfaces in metric spaces},
J. Math. Pures Appl. (9) 93 (2010), no. 6, 599--622.

\bibitem{KR}J. Kristensen and F. Rindler,
\textit{Piecewise affine approximations for functions of bounded variation},
Numer. Math. 132 (2016), no. 2, 329--346. 

\bibitem{L-Fed}P. Lahti,
\textit{A Federer-style characterization of sets of finite perimeter on metric spaces},
Calc. Var. Partial Differential Equations, October 2017, 56:150.

\bibitem{L-NC}P. Lahti,
\textit{A new Cartan-type property and strict quasicoverings when $p=1$ in metric spaces},
to appear in Ann. Acad. Sci. Fenn. Math.
https://arxiv.org/abs/1801.09572

\bibitem{L-FC}P. Lahti,
\textit{A notion of fine continuity for BV functions on metric spaces},
Potential Anal. 46 (2017), no. 2, 279--294.

\bibitem{L-LSC}P. Lahti,
\textit{Quasiopen sets, bounded variation and lower semicontinuity in metric spaces},
preprint 2017.
https://arxiv.org/abs/1703.04675

\bibitem{L-CK}P. Lahti,
\textit{The Choquet and Kellogg properties for the fine topology when $p=1$ in metric spaces},
preprint 2017.
https://arxiv.org/abs/1712.08027

\bibitem{L-SA}P. Lahti,
\textit{Strong approximation of sets of finite perimeter in metric spaces},
manuscripta mathematica, March 2018, Volume 155, Issue 3–4, pp 503--522.

\bibitem{L-ZB}P. Lahti,
\textit{The variational 1-capacity and BV functions with zero boundary values on metric spaces},
preprint 2017.
https://arxiv.org/abs/1708.09318

\bibitem{LaSh}P. Lahti and N. Shanmugalingam,
\textit{Fine properties and a notion of quasicontinuity for $\BV$ functions on metric spaces},
Journal de Math\'ematiques Pures et Appliqu\'ees, Volume 107, Issue 2, February
2017, Pages 150--182.

\bibitem{MZ}J. Mal\'{y} and W. Ziemer,
\textit{Fine regularity of solutions of elliptic partial differential equations},
Mathematical Surveys and Monographs, 51. American Mathematical Society, Providence, RI, 1997. xiv+291 pp.

\bibitem{M}M.~Miranda, Jr.,
\textit{Functions of bounded variation on ``good'' metric spaces},
J. Math. Pures Appl. (9) 82  (2003),  no. 8, 975--1004.

\bibitem{S2}N. Shanmugalingam,
\textit{Harmonic functions on metric spaces},
Illinois J. Math. 45 (2001), no. 3, 1021--1050.

\bibitem{S}N. Shanmugalingam,
\textit{Newtonian spaces: An extension of Sobolev spaces to metric measure spaces},
Rev. Mat. Iberoamericana 16(2) (2000), 243--279.

\bibitem{Zie89}W. P. Ziemer,
\textit{Weakly differentiable functions. Sobolev spaces and functions of bounded variation},
Graduate Texts in Mathematics, 120. Springer-Verlag, New York, 1989. 

\end{thebibliography}
\end{document}